\documentclass[11pt]{article}

\usepackage{amstext,amssymb,amsmath,amsbsy}
\usepackage{xcolor}
\usepackage{hyperref}
\usepackage{amscd}
\usepackage{amsfonts}
\usepackage{indentfirst}
\usepackage{verbatim}
\usepackage{amsmath}
\usepackage{amsthm}
\usepackage{enumerate}
\usepackage{graphicx}
\usepackage{color, soul}
\usepackage[OT1]{fontenc}
\usepackage[latin1]{inputenc}
\usepackage[english]{babel}
\usepackage{amssymb}
\usepackage{subfig}
\usepackage{algorithm}
\usepackage{algpseudocode}

\usepackage{mathtools}

\newcommand{\R}{\mathbb{R}}

\newcommand{\ds}{\displaystyle}

\newcommand{\x}{{\bf x}}

\newcommand{\p}{{\bf p}}

\newcommand{\Div}{{\rm div}}

\newtheorem{Theorem}{Theorem}[section]
\newtheorem{Lemma}{Lemma}[section]

\newtheorem{Corollary}{Corollary}[section]
\newtheorem{remark}{Remark}[section]

\newtheorem*{Assumption*}{Assumption}

\newtheorem{problem}{Problem}[section]
\newtheorem*{problem*}{Problem}
\numberwithin{equation}{section}

\begin{document}

\title{Determining initial conditions for nonlinear hyperbolic equations with time dimensional reduction and the Carleman contraction principle}

\author{Trong D. Dang\thanks{Faculty of Mathematics and Computer Science, University of Science, Vietnam National University, Ho Chi Minh City, Vietnam, \texttt{ddtrong@hcmus.edu.vn}.} \and Loc H. Nguyen\thanks{Department of Mathematics and Statistics, University of North Carolina at
Charlotte, Charlotte, NC, 28223, USA, \texttt{loc.nguyen@charlotte.edu}, corresponding author.} 
\and
Huong T. T. Vu\thanks{Department of Information Technology, University of Finance-Marketing, Ho Chi Minh City, Vietnam, \texttt{vtthuong@ufm.edu.vn}.}}


\date{}
\maketitle
\begin{abstract}
This paper aims to determine the initial conditions for quasi-linear hyperbolic equations that include nonlocal elements. We suggest a method where we approximate the solution of the hyperbolic equation by truncating its Fourier series in the time domain with a polynomial-exponential basis. This truncation effectively removes the time variable, transforming the problem into a system of quasi-linear elliptic equations. We refer to this technique as the ``time dimensional reduction method."
To numerically solve this system comprehensively without the need for an accurate initial estimate, we used the newly developed Carleman contraction principle. We show the efficiency of our method through various numerical examples.
The time dimensional reduction method stands out not only for its precise solutions but also for its remarkable speed in computation.
\end{abstract}


\noindent{\it AMS subject classification: 	35R30;
65M32.

}

\section{Introduction} \label{sec intr}

Let $\Omega$ be an open and bounded domain of $\R^d.$
Assume the $\Omega$ has a smooth boundary.
Let $T$ be a positive number that represents the final time.
Let $\mathcal F$ be an operator acting on $C^2(\overline \Omega \times [0, T])$ defined as
\begin{equation}
	\mathcal F(v)(\x, t) = F\Big(\x, t, v, \nabla v, v_t, \int_0^t K(s)v(\x, s)ds\Big)
	\label{1.1}
\end{equation}
for all $(\x, t) \in \overline \Omega \times [0, T]$ 
where the kernel $K: \R \to \R$ 
and the map 
$F: \overline \Omega \times [0, T] \times \R \times \R^d \times \R \times \R \to \R$ are given.
Let $u = u(\x, t)$ be the solution to the following initial value problem
\begin{equation}
	\left\{
	\begin{array}{rcll}
		u_{tt} &=& \Delta u + \mathcal F(u) & (\x, t) \in \Omega \times (0, T),\\
		u(\x, t) &=& 0 &(\x, t) \in \partial\Omega \times (0, T),\\
		u(\x, 0) &=& g(\x) &\x \in \Omega,\\
		u_t(\x, 0) &=& 0 &\x \in \Omega.
	\end{array}
	\right.
	\label{main}
\end{equation}
Here,
 $g: \overline \Omega \to \R$ is a smooth function.
The inverse problem we are interested in is formulated as follows.
\begin{problem}[An inverse source problem for quasi-linear nonlocal hyperbolic equation]
	Given the boundary data
	\begin{equation}		
		h(\x, t) = \partial_{\nu} u(\x, t) 
		\quad 
		\mbox{for all } (\x, t) \in \partial \Omega \times (0, T),
		\label{data}
	\end{equation}	 
	compute the function $g(\x)$ for $\x \in \Omega$.
	\label{isp}
\end{problem}

Identifying the optimal conditions for the existence and uniqueness of solutions to a complex problem like \eqref{main} is a substantial challenge. This paper does not delve into exploring these aspects for \eqref{main}; rather, these aspects are accepted as given assumptions. For completeness, we present a set of conditions for $\p$ and $F$ that ensure the existence and regularity of solutions to \eqref{main}. We assume that $\p$ is smooth and has support compactly contained in $\Omega$ and that $F$ is independent of the first-order derivatives of $u$. Additionally, assume that $|\mathcal F u|$ is bounded by $C_1|u| + C_2$ for any function $u$ within $H^2(\Omega \times (0, T))$, where $C_1$ and $C_2$ represent positive constants. Under these conditions, the unique resolution of \eqref{main} can be guaranteed by applying \cite[Theorem 10.14]{Brezis:Springer2011}, originally proven by Lions in \cite{Lions:sp1972}. For more insights on the feasibility of \eqref{main} in linear contexts, references such as \cite{Evans:PDEs2010, Ladyzhenskaya:sv1985} are recommended.

Problem \ref{isp} represents a nonlinear version of a crucial issue in biomedical imaging, specifically thermo/photo-acoustic tomography. The experimental procedure for this problem includes applying non-ionizing laser pulses or microwaves to the biological tissue being examined, such as in mammography of a woman's breast, as explained in \cite{Krugeretal:mp1995, Krugerelal:mp1999, Oraevskyelal:ps1994}. Part of this energy gets absorbed and transforms into heat, causing thermal expansion and generating ultrasonic waves. Capturing these ultrasonic pressure waves on a surrounding surface gives insights into the tissue's internal structure.
Current solutions for thermo/photo-acoustic tomography are largely based on linear hyperbolic equations and involve methods like explicit reconstruction in open space \cite{DoKunyansky:ip2018, Haltmeier:cma2013, Natterer:ipi2012, Linh:ipi2009}, time reversal \cite{KatsnelsonNguyen:aml2018, Hristova:ip2009, HristovaKuchmentLinh:ip2006, Stefanov:ip2009, Stefanov:ip2011}, quasi-reversibility \cite{ClasonKlibanov:sjsc2007, LeNguyenNguyenPowell:JOSC2021}, and iterative approaches \cite{Huangetal:IEEE2013, Paltaufetal:ip2007, Paltaufetal:osa2002}. These studies focus on simple models of non-damping, isotropic media. For more intricate models with damping or attenuation terms, refer to \cite{Ammarielal:sp2012, Ammarietal:cm2011, Haltmeier:jmiv2019, Acosta:jde2018, Burgholzer:pspie2007, Homan:ipi2013, Kowar:SISI2014, Kowar:sp2012, Nachman1990}.
The paper \cite{NguyenKlibanov:ip2022} addresses the nonlinear aspect, using Carleman estimates and the contraction principle to manage nonlinearity. Our approach, similar to \cite{NguyenKlibanov:ip2022}, also utilizes these techniques to resolve Problem \ref{isp}. However, this paper introduces a novel element compared to \cite{NguyenKlibanov:ip2022}: the time-dimensional reduction, which involves truncating the Fourier series in time. This enhancement allows our solver to handle a wider range of nonlinearities and significantly boosts computational efficiency.
Additionally, the time-reduction method improves the convergence rate against noise. While the convergence rate in \cite{NguyenKlibanov:ip2022} is of H\"older type, reducing the time domain regularizes the inverse problem, leading to Lipschitz-type convergence.

Traditional approaches for addressing nonlinear inverse problems predominantly rely on optimization techniques. These methods are local, providing effective solutions when there are accurate initial solution estimates. 
Addressing nonlinear problems without requiring a reliable initial guess poses an intriguing, challenging, and significant issue in science. A widely recognized strategy for globally solving nonlinear inverse problems is known as convexification. The key principle of convexification is to employ suitable Carleman weight functions to make the mismatch functional convex, a process thoroughly supported by the established Carleman estimates. Since its invention in \cite{KlibanovIoussoupova:SMA1995}, various adaptations of the convexification approach have emerged \cite{KlibanovNik:ra2017, KhoaKlibanovLoc:SIAMImaging2020, Klibanov:sjma1997, Klibanov:nw1997, Klibanov:ip2015, KlibanovKolesov:cma2019, KlibanovLiZhang:ip2019, KlibanovLeNguyenIPI2022, KlibanovLiZhang:SIAM2019, LeNguyen:JSC2022}. Notably, it has been effectively applied in experimental settings for inverse scattering problems in the frequency domain using only backscattering data \cite{VoKlibanovNguyen:IP2020, Khoaelal:IPSE2021, KlibanovLeNguyenIPI2022}. However, a notable drawback of the convexification method is its time-consuming.
To address this, we leverage the techniques from \cite{LeCON2023, LeNguyen:jiip2022, Nguyen:AVM2023} to introduce a novel method that combines Fourier expansion, fixed-point iteration, the contraction principle, and an appropriate Carleman estimate for globally solving nonlinear inverse problems. By ``global," we mean our method does not require an initial guess close to the true solution.

Our method involves the following main ingredients:
\begin{enumerate}
	\item {\bf Time reduction}.  We propose to truncate the Fourier series of the solution to \eqref{main} using a polynomial-exponential basis in time. This crucial step converts a problem from $d + 1$ dimensions to one with $d$ dimensions, significantly boosting computational speed. It also allows us to address Problem \ref{isp} even when the nonlinearity $F$ involves a memory term, like the Volterra integral mentioned in \eqref{1.1}.
	\item {\bf Fixed point-like iteration}. After minimizing the time domain, we can derive a system of quasi-linear elliptic equations. We tackle this system by reformulating it into the equation $\Phi(x) = x$, where $\Phi$ is an operator incorporating specific Carleman weight functions. Using Carleman estimates, we demonstrate that $\Phi$ acts as a contraction mapping, leading us directly to the solution of the inverse problem through the calculated fixed point of $\Phi$.
\end{enumerate}

The primary advantages of our approach are:
\begin{enumerate}
	\item It does not necessitate an accurate initial guess.
	\item It is broadly applicable, as it does not impose specific structures on the nonlinearity $F$.
	\item It achieves a rapid rate of convergence.
\end{enumerate}

The paper is organized as follows. 
In Section \ref{sec_app}, we present the time dimensional reduction method.
In Section \ref{sec_contraction}, we construct a map and show that this map is contractive. Section \ref{sec4} demonstrates how the fixed point of the aforementioned contraction mapping approximates the true solution.
Section \ref{num} is to present some numerical examples.
Section \ref{sec6} is for concluding remarks.

\section{The time dimensional reduction model}\label{sec_app}

	The polynomial-exponential basis $\{\Psi_n\}_{n \geq 1}$, as introduced in \cite{Klibanov:jiip2017} and extended for higher dimensions in \cite{NguyenLeNguyenKlibanov:2023}, plays the key role in our time reduction method. This basis is constructed in the following manner:
For every $n \geq 1$, we define a function $\phi(t) = t^{n-1} e^{t - T/2}$, where $t$ is in the interval $(0, T)$. It is clear that the set $\{\phi_n\}{n \geq 1}$ is a complete set in the space $L^2(0, T)$. When we apply the Gram-Schmidt orthogonalization process to $\{\phi_n\}{n \geq 1}$, we obtain an orthonormal basis $\{\Psi_n\}{n \geq 1}$ of $L^2(0, T)$. An important characteristic of this basis $\{\Psi_n\}{n \geq 1}$ in $L^2(0, T)$ is that for every $n$, the derivative $\Psi_n'(t)$ is non-zero. This particular feature is crucial in the application of the time-dimensional reduction method, which will be elaborated upon later in the paper.	
	
	In order to propose a numerical method for solving Problem \ref{isp}, it is permissible to approximate the wave function $u(\x, t)$ by truncating its Fourier series as follows:
	\begin{equation}
		u(\x, t) = \sum_{n = 1}^{\infty} u_n(\x) \Psi_n(t)
		\simeq \sum_{n = 1}^{N} u_n(\x) \Psi_n(t)
		\label{2.1}
	\end{equation}
	where $N$ is a cutoff number determined by the given data, see Section \ref{num} for our data-driven method to find $N$. 
	The Fourier coefficient $u_n$ in \eqref{2.1} is given by
	\begin{equation}
		u_n(\x) = \int_{0}^T u(\x, t) \Psi_n(t) dt, 
		\quad
		n \geq 1.
	\end{equation}
By substituting the approximation \eqref{2.1} into the governing hyperbolic equation in \eqref{main}, we obtain the following equation
\begin{equation}
	\sum_{n = 1}^{N} u_n(\x) \Psi_n''(t) = \sum_{n = 1}^{N} \Delta  u_n(\x) \Psi_n(t) + \mathcal F\Big(\sum_{n = 1}^{N} u_n(\x) \Psi_n(t)\Big)
	\label{2.3}
\end{equation}
for $(\x, t) \in \Omega \times (0, T).$
By multiplying both sides of \eqref{2.3} by $\Psi_m(t)$ for each $m \in \{1, \dots, N\}$ and integrating the resulting equation, we arrive at
\begin{equation}
	\sum_{n = 1}^{N} s_{mn} u_n(\x) = \Delta u_m(\x) + {\bf f}_mU,
	\quad
	m \in \{1, 2, \dots, N\}
	\label{2.4}
\end{equation}
for  $\x \in \Omega$,
where 
\begin{align*}
	U(\x) &=(u_1, u_2, \dots, u_N)^{\rm T},\\
	{\bf f}_mU(\x) &=  \mathcal F\Big(\sum_{n = 1}^{N} u_n(\x) \Psi_n(t)\Big) \Psi_m(t)dt,\\
	s_{mn} &= \int_0^T \Psi_n''(t) \Psi_m(t) dt.
\end{align*}
Defining ${\bf F}U = ({\bf f}_1 U,  \dots, {\bf f}_N U)^{\rm T}$ and $S = (s_{mn})_{m,n = 1}^N$, we can rewrite the system \eqref{2.4} as the form
\begin{equation}
	\Delta U(\x) - SU(\x) + {\bf F}U(\x) = 0 
	\quad 
	\mbox{for all } \x \in \Omega.
	\label{2.5}
\end{equation}

We next derive the boundary conditions for the vector $U$.
It follows from the Dirichlet boundary condition in \eqref{main} that for each $m \in \{1, \dots, N\}$ that for all $\x \in \partial \Omega,$
\begin{equation}
	u_m(\x) = 0.
	\label{2.6}
\end{equation}
We can compute the Neumann condition $\partial_{\nu} U(\x)$ from the given boundary data \eqref{data} as 
\begin{equation}
	\partial_{\nu} u_m(\x) = \int_{0}^T h(\x, t) \Psi_m(t) dt
	\label{2.7}
\end{equation}
for all $m \in \{1, \dots, N\}$, $\x \in \partial \Omega$.

In summary, denote by
\begin{equation}
	{\bf h}(\x) =  \Big(\ds\int_{0}^T h(\x, t) \Psi_m(t) dt\Big)_{m = 1}^N,
	\quad \x \in \partial \Omega
	\label{indirectdata}
\end{equation}
the indirect data that can be computed from the given data in \eqref{data}.
By \eqref{2.5}, \eqref{2.6} and \eqref{2.7}, we have derived a system of quasi-linear elliptic equations for the vector $U$
\begin{equation}
	\left\{
		\begin{array}{ll}
			\Delta U(\x) - SU(\x) + {\bf F}U(\x) = 0 &\x \in \Omega,\\
			U(\x) = 0 &\x \in \partial \Omega,\\
			\partial_{\nu} U(\x) = {\bf h}(\x) & \x \in \partial \Omega.
		\end{array}
	\right.
	\label{2.8}
\end{equation}

The inverse source problem under consideration is reduced to the problem solving \eqref{2.8}.
Having the computed solution $U^{\rm comp} = (u_1^{\rm comp}, \dots, u^{\rm comp}_N)$ to \eqref{2.8} in hand, due to \eqref{2.1}, we can find the computed source function $g^{\rm comp}$ by 
\begin{equation}
	g^{\rm comp}(\x) = u^{\rm comp}(\x, 0) =  \sum_{n = 1}^{N} u_n(\x) \Psi_n(0)
	\quad
	\mbox{for all } \x \in \Omega.
	\label{2.9}
\end{equation}
\begin{remark}
The technique we employ, which entails removing the time variable $t$ from equation \eqref{main} to obtain \eqref{2.8}, is known as the ``time-dimensional reduction method." 
The quasi-linear system \eqref{2.8} is called the ``time-reduction model" associated with \eqref{main}.
By eliminating the dimension associated with the time variable, we achieve a substantial reduction in computational costs. In fact, we solve problem \eqref{2.8} in $d$ dimensions, while the initial problem \eqref{main}, prior to time reduction, exists in $d+1$ dimensions, consisting of $d$ spatial dimensions along with the time variable.
\label{rm2.1}
\end{remark}

\begin{remark}[The significance of the polynomial-exponential basis]
	Here, we delve into the significance of the polynomial-exponential basis used in the time reduction method. A pertinent question might be why the basis $\{\Psi_n\}_{n \geq 1}$ was specifically chosen for the Fourier expansion in equation \eqref{2.1} over numerous other possibilities. The rationale for this choice lies in the limitations of more commonly used bases, like Legendre polynomials or trigonometric functions, which may not be appropriate for this particular application.
The complication with these conventional bases is that their initial function is a constant, which has zero derivative. This characteristic causes the Fourier coefficient $u_{1}(\x)$ to be excluded in the sums $\sum_{n = 1}^N u_{n}(\x) \Psi_n''(t)$ in equation \eqref{2.3}, adversely affecting accuracy. In contrast, the polynomial-exponential basis used in our study is specifically chosen for equation \eqref{2.3} because it fulfills the crucial requirement that $\Psi_n'$, for each $n \geq 1$, is non-zero.

The superior effectiveness of the polynomial-exponential basis has been validated in various research works, including \cite{LeNguyenNguyenPowell:JOSC2021, NguyenLeNguyenKlibanov:2023}. In \cite{LeNguyenNguyenPowell:JOSC2021}, we compared the polynomial-exponential basis with the traditional trigonometric basis in expanding wave fields for problems in photo-acoustic and thermo-acoustic tomography. The outcomes distinctly showed the polynomial-exponential basis's enhanced performance. Moreover, in \cite{NguyenLeNguyenKlibanov:2023}, this basis was employed in the Fourier expansion for computing derivatives of noise-corrupted data. The results revealed that the polynomial-exponential basis was more accurate compared to the trigonometric basis, particularly when term-by-term differentiation of the Fourier series was necessary to resolve ill-posed problems.

Thus, the polynomial-exponential basis emerges as the preferred choice when term-by-term differentiation of Fourier expansions is required, thanks to its beneficial properties and proven effectiveness.
\end{remark}

\begin{remark}
	Recall that upon solving the time-reduction model \eqref{2.8}, the solution to Problem \ref{isp} can be computed using \eqref{2.9}. As the time-reduction model \eqref{2.8} relies on the cutoff number $N$, studying the convergence of our method as $N$ tends to infinity poses significant challenges. However, addressing this question is beyond the scope of this paper, as our primary focus lies in the computational aspects.
\end{remark}

There are several methods to solve the time-reduction model \eqref{2.8}. We list some options below
\begin{enumerate}
	\item The least squares optimization technique: This approach is popular in scientific research, involving the minimization of a cost functional. One takes the global minimizer as the solution. For instance, consider the functional
	\[
	V \mapsto J(v) := \int_{\Omega} |\Delta V(\x) - SV(\x) + {\bf F}V(\x)|^2 d\x 
	+ \mbox{some regularization terms}
\]
subject to the boundary conditions $V = {\bf f}$ and $\partial_{\nu}V = {\bf h}$ on $\partial \Omega$. Finding the global minimizer of $J$ can be challenging due to the potential existence of multiple local minimizers. To obtain reliable results with this optimization method, it is typically necessary to start with a well-informed initial guess. 
	\item The convexification method:  Initially introduced in \cite{KlibanovIoussoupova:SMA1995}, this method allows for solving inverse problems without necessitating a good initial guess. Various versions of the convexification method have been developed and explored in subsequent papers \cite{KlibanovNik:ra2017, VoKlibanovNguyen:IP2020, Khoaelal:IPSE2021, KhoaKlibanovLoc:SIAMImaging2020, Klibanov:jiip2017, Klibanov:ip2020, KlibanovNguyenTran:JCP2022, LeLeNguyen:Arxiv2022, LeNguyen:JSC2022}.
	\item The Carleman contraction method \cite{Nguyen:AVM2023}  along with the recently developed Carleman-Newton method \cite{AbhishekLeNguyenKhan, LeNguyenTran:CAMWA2022}, can provide reliable solutions to \eqref{2.8} without requiring an initial guess. These methods offer faster convergence rates compared to the convexification method. In this paper, we generalize the Carleman contraction method in \cite{Nguyen:AVM2023} to solve \eqref{2.8}. This will be presented in Section \ref{sec_contraction}. 
\end{enumerate}

The key aspect shared by the convexification, Carleman contraction, and Carleman-Newton methods is the inclusion of Carleman weight functions within the procedure. The efficacy of these approaches is demonstrated using Carleman estimates.

\section{The Carleman contraction mapping}
\label{sec_contraction}

Let $p > \lceil d/2 \rceil + 2$ be an integer.
Define the functional space 
\begin{equation}
	H = \big\{\varphi \in H^p(\Omega)^N: \varphi|_{\partial \Omega} = 0\big\}.
\end{equation}
It is obvious that $H$ is a closed subspace of the Hilbert space $H^p(\Omega)^N.$
Let $\x_0$ be a point in $\mathbb{R}^d \setminus \overline \Omega$ such that $r(\x) = \vert\x - \x_0\vert > 1$ for all $\x \in \Omega$. 
For each $\epsilon > 0$, $\lambda > 0$ and $\beta > 0$, introduce the map
\begin{equation}
	\Phi_{\lambda, \beta, \epsilon}: H \to H,
	\quad
	W \mapsto \Phi_{\lambda, \beta, \epsilon}(W) =  \underset{V \in H}{\mbox{argmin}} J_{\lambda, \beta, \epsilon}^W(V)
	\label{4.3}
\end{equation}
where
\begin{equation}
	J_{\lambda, \beta, \epsilon}^W(V) 
	= \int_{\Omega}e^{2\lambda r^{-\beta}}|\Delta V - S V + {\bf F}(W)|^2d\x 
	+ \lambda^2 \int_{\partial \Omega} e^{2\lambda r^{-\beta}} |\partial_{\nu}V - {\bf h}|^2d\sigma(\x)
	+ \epsilon \|V\|_{H}^2
	\label{4.4}
\end{equation}
for all $V \in H.$
The existence and uniqueness of $J_{\lambda, \beta, \epsilon}$ is obvious. 
For each fix $W \in H$, $J_{\lambda, \beta, \epsilon}^W(V)$ is strictly convex in Hilbert space $H$.
In other words, the map $\Phi_{\lambda, \beta, \epsilon}$ is well-defined, see also \cite[Remark 3.1]{Nguyen:AVM2023}.
As mentioned in the last paragraph of Section \ref{sec_app}, the presence of the Carleman weight function $e^{2\lambda r^{-\beta}}$ in the integral in the right-hand side of \eqref{4.4} is the main point of our globally convergent method to compute a numerical solution to \eqref{2.8}.

Assume that $F$ is Lipschitz; i.e., there is a constant $M_F$ such that
\begin{equation}
	|F(\x, t, s_1, \p_1, r_1, l_1) - F(\x, t, s_2, \p_2, r_2, l_2)|
	\leq
	M_F (|s_1 - s_2| + |\p_1 - \p_2| + |r_1 - r_2| + |l_1 - l_2|)
	\label{Lipschitz}
\end{equation}
for all $\x \in \overline \Omega,$ $s_i \in \R$, $\p_i \in \R^d$, $r_i \in \R$, $l_i \in \R$, $i \in \{1, 2\}.$
Then, we can find a constant ${\bf M}$ depending on $M_F$, $K$, $\{\Psi_n\}_{n = 1}^N$, $N$ and $T$ such that 
\begin{equation}
	|{\bf F}(V_1) - {\bf F}(V_2)|^2
	\leq {\bf M}\big(|V_1 - V_2|^2 + |\nabla V_1 - \nabla V_2|^2\big)
	\label{LipschitzF}
\end{equation}
for all vector-valued functions $V_1$ and $V_2$ and for all $(\x, t) \in \overline \Omega \times [0, T].$
For $\lambda > 0$, $\beta > 0$, and $\varepsilon > 0$, define the norm
$\|\cdot\|_{\lambda, \beta, \epsilon}$ 
\begin{equation}
	\|U\|_{\lambda, \beta, \epsilon} = 
	\left(
		\int_{\Omega} e^{2\lambda r^{-\beta}} 
			(\lambda^2|U|^2 + |\nabla U|^2)d\x
		+
		\lambda\int_{\partial \Omega} e^{2\lambda r^{-\beta}} |\nabla U|^2 d\sigma(\x)	 
	+ \frac{\epsilon}{\lambda} \|U\|_{H}^2
	\right)^{1/2} 
	\label{4,2}
\end{equation}	
for all $U \in H$.
We have the theorems.

\begin{Theorem}
Assume \eqref{Lipschitz} and hence \eqref{LipschitzF} hold true.
Then, there is a constant $\beta_0$ depending only on ${\bf M}$, $d$, $\Omega,$ and $\x_0$ such that for all $\beta \geq \beta_0$, we have
\begin{equation}
		\|\Phi_{\lambda, \beta, \epsilon}(U) - \Phi_{\lambda, \beta, \epsilon}(V)\|_{\lambda, \beta, \epsilon}
		\leq \sqrt{\frac{C}{\lambda}} \|U - V\|_{\lambda, \beta, \epsilon}
		\label{4,1}
	\end{equation}
	for all $\lambda > \lambda_0$
where $\lambda_0 = \lambda_0({\bf M}, d, \Omega, d, \x_0, \beta)$ and $C = C({\bf M}, d, \Omega, d, \x_0, \beta)$ depending only on the listed parameters.
Consequently, choosing $\lambda > \lambda_0$ sufficiently large, the map $\Phi_{\lambda, \beta, \epsilon}: H \to H$ is a contraction mapping with respect to the norm $\|\cdot\|_{\lambda, \beta, \epsilon}$.
\label{thm_contract}
\end{Theorem}

The following result is a point-wise Carleman estimate, which plays a pivotal role in the proof of Theorem \ref{thm_contract} and the convergence of our method with respect to the noise in the next section. 
\begin{Lemma}[See Theorem 3.1 in \cite{LeLeNguyen:Arxiv2022}]
    Let $\lambda > 0$ and $v \in C^2(\overline \Omega).$
    Then, there exists a positive constant $\beta_0$ depending only on $\x_0,$ $d$ and $\Omega$ 
    such that for all $\beta \geq \beta_0$ and $\lambda \geq \lambda_0 = 2R^\beta$, where $R = \max_{\x \in \overline \Omega}\{|\x - \x_0|\}$, we have
    \begin{equation}
        r^{\beta + 2} e^{2\lambda r^{-\beta}} |\Delta v|^2  
    \geq C\Big[
        \Div (V) + \lambda^3\beta^4e^{2\lambda r^{-\beta}} r^{-2\beta - 2} |v|^2
        + \lambda \beta e^{2\lambda r^{-\beta}} |\nabla v|^2 
    \Big].
    \label{CarEst}
    \end{equation}
    Here, $V$ is  a vector-valued function satisfying
    \begin{equation}
        |V| \leq Ce^{2\lambda r^{-\beta}}(\lambda^3 \beta^3 r^{-2\beta - 2}|v|^2 + \lambda \beta |\nabla v|^2)
        \label{divU}
    \end{equation} and $C$ is a constant depending only on $\x_0,$ $\Omega$, and $d$. 
    \label{lemCarpointwise}
\end{Lemma}
We refer the reader to \cite[Theorem 3.1]{LeLeNguyen:Arxiv2022} for the proof of this Carleman estimate.
A direct consequence of \eqref{CarEst} is as follows.
\begin{Corollary}
Let $v \in C^2(\overline \Omega)$ with $v|_{\partial \Omega} = 0.$ We have
\begin{equation}
	\int_{\Omega} e^{2\lambda r^{-\beta}} |\Delta v|^2 d\x
	\geq
	-C \lambda \int_{\partial\Omega}e^{2\lambda r^{-\beta}}   |\nabla v|^2 d\sigma(\x)
	\\
	+C \int_{\Omega} e^{2\lambda r^{-\beta}} \big[
		 \lambda^3  |v|^2
        + \lambda   |\nabla v|^2
		 \big]d\x
		 \label{4.3333}
	\end{equation} 
	for $\lambda, \beta$ and $C$ as in Lemma \ref{lemCarpointwise}.
\end{Corollary}
\begin{proof}
	Integrating \eqref{CarEst} and using integrating by parts, we have
	\[
		\int_{\Omega}r^{\beta + 2} e^{2\lambda r^{-\beta}} |\Delta v|^2 
		\geq 
		 C\int_{\partial\Omega} V\cdot \nu d\x
		 + C \int_{\Omega}\big[
		 \lambda^3\beta^4e^{2\lambda r^{-\beta}} r^{-2\beta - 2} |v|^2
        + \lambda \beta e^{2\lambda r^{-\beta}} |\nabla v|^2
		 \big].
	\]
	By \eqref{divU}, we have
	\begin{multline*}
		\int_{\Omega}r^{\beta + 2} e^{2\lambda r^{-\beta}} |\Delta v|^2 
		\geq 
		 -C\int_{\partial\Omega}e^{2\lambda r^{-\beta}}(\lambda^3 \beta^3 r^{-2\beta - 2}|v|^2 + \lambda \beta |\nabla v|^2) d\sigma(\x)
		 \\
		 + C \int_{\Omega}\big[
		 \lambda^3\beta^4e^{2\lambda r^{-\beta}} r^{-2\beta - 2} |v|^2
        + \lambda \beta e^{2\lambda r^{-\beta}} |\nabla v|^2
		 \big].
	\end{multline*}
	Since $v|_{\partial \Omega} = 0$, we have
	\begin{multline*}
	\int_{\Omega}r^{\beta + 2} e^{2\lambda r^{-\beta}} |\Delta v|^2 d\x
	\geq
	-C\int_{\partial\Omega}e^{2\lambda r^{-\beta}} \lambda \beta |\nabla v|^2 d\sigma(\x)
	\\
	+C \int_{\Omega}\big[
		 \lambda^3\beta^4e^{2\lambda r^{-\beta}} r^{-2\beta - 2} |v|^2
        + \lambda \beta e^{2\lambda r^{-\beta}} |\nabla v|^2
		 \big],
	\end{multline*}
	which directly implies \eqref{4.3333}. 
	 It should be noted that in the above argument, the constant $C$ is allowed to depend on $\beta$, $\Omega$, and $\x_0$.
\end{proof}

 \begin{proof}[Proof of Theorem \ref{thm_contract}]
 	Define $U_1 = \Phi_{\lambda, \beta, \epsilon}(U).$
	By the definition of  $\Phi_{\lambda, \beta, \epsilon}$ in \eqref{4.3}, we have $U_1$ is the minimizer of $J_{\lambda, \beta, \epsilon}^{U}$ in $H$ where $J_{\lambda, \beta, \epsilon}^{U}$ is defined in \eqref{4.4} with $U$ replacing $W$.
	By the variational principle, for all $\varphi \in H,$ we have
\begin{multline}
	\langle
		e^{2\lambda r^{-\beta}} (\Delta U_{1} - S U_{1} + {\bf F}(U)),
		\Delta \varphi - S \varphi
	\rangle_{L^2(\Omega)^N}
		\\
	+ \lambda^2 
	\langle
		e^{2\lambda r^{-\beta}} (\partial_{\nu}U_{1} - {\bf h}), \partial_{\nu} \varphi		
	\rangle_{L^2(\partial\Omega)^N}
	+\epsilon
	\langle
		 U_{1}, \varphi		
	\rangle_{H} = 0.
	\label{4,3}
\end{multline}
 Similarly, set $V_1 = \Phi_{\lambda, \beta, \epsilon}(V)$. We have
 \begin{multline}
	\langle
		e^{2\lambda r^{-\beta}} (\Delta V_{1} - S V_{1} + {\bf F}(V)),
		\Delta \varphi - S \varphi
	\rangle_{L^2(\Omega)^N}
		\\
	+ \lambda^2 
	\langle
		e^{2\lambda r^{-\beta}} (\partial_{\nu}V_{1} - {\bf h}), \partial_{\nu} \varphi		
	\rangle_{L^2(\partial\Omega)^N}
	+\epsilon
	\langle
		 V_{1}, \varphi		
	\rangle_{H} = 0
	\label{4,4}
\end{multline}
for all $\varphi \in H.$
Subtracting \eqref{4,4} from \eqref{4,3}  gives
 \begin{multline}
	\langle
		e^{2\lambda r^{-\beta}} (\Delta (U_1 - V_{1}) - S (U_1 - V_{1}) + {\bf F}(U) - {\bf F}(V)),
		\Delta \varphi - S \varphi
	\rangle_{L^2(\Omega)^N}
		\\
	+ \lambda^2 
	\langle
		e^{2\lambda r^{-\beta}} \partial_{\nu}(U_{1} - V_1), \partial_{\nu} \varphi		
	\rangle_{L^2(\partial\Omega)^N}
	+\epsilon
	\langle
		 U_1 - V_{1}, \varphi		
	\rangle_{H} = 0
	\label{4,5}
\end{multline}
for all $\varphi \in H.$
In particular, using $\varphi = U_1 - V_1$ in \eqref{4,5} gives
\begin{align*}
	\int_{\Omega} e^{2\lambda r^{-\beta}} |\Delta \varphi - S\varphi|^2 d\x
	+ &\lambda^2 \int_{\partial \Omega} e^{2\lambda r^{-\beta}} |\partial_{\nu}\varphi|^2d\sigma(\x)
	+ \epsilon \|\varphi\|_H^2 
	\\
	&= \int_{\Omega} e^{2\lambda r^{-\beta}} (\Delta \varphi - S\varphi) ({\bf F}(U) - {\bf F}(V))d\x 
	\\
	&\leq \frac{1}{2} \int_{\Omega} e^{2\lambda r^{-\beta}} |\Delta \varphi - S\varphi|^2 d\x +  \frac{1}{2} \int_{\Omega} e^{2\lambda r^{-\beta}} |{\bf F}(U) - {\bf F}(V)|^2d\x.
\end{align*}
Hence,
\begin{equation}
	\int_{\Omega} e^{2\lambda r^{-\beta}} |\Delta \varphi - S\varphi|^2 d\x
	+\lambda^2 \int_{\partial \Omega} e^{2\lambda r^{-\beta}} |\partial_{\nu}\varphi|^2d\sigma(\x)
	+ \epsilon \|\varphi\|_H^2 
	\leq C\int_{\Omega} e^{2\lambda r^{-\beta}} |{\bf F}(U) - {\bf F}(V)|^2d\x.
	\label{3.1414}
\end{equation}
It follows from \eqref{3.1414} and the inequality $(a - b)^2 \leq \frac{1}{2}a^2 - b^2$ that
\begin{multline}
	\int_{\Omega} e^{2\lambda r^{-\beta}} |\Delta \varphi|^2 d\x
	+ \lambda^2 
	\int_{\partial \Omega}  e^{2\lambda r^{-\beta}} |\partial \varphi|^2 d\sigma(\x)
		+\epsilon
	\|
		 \varphi		
	\|_{H}^2  
	\\
	\leq
	C\Big( \int_{\Omega}e^{2\lambda r^{-\beta}} |S \varphi|^2d\x
	+
	\int_{\Omega}e^{2\lambda r^{-\beta}} |{\bf F}(U) - {\bf F}(V)|^2d\x
	\Big).
	\label{4,6}
\end{multline}
Combining Carleman estimate \eqref{4.3333} for $\varphi$ and \eqref{4,6}, we have
\begin{multline}
-C \lambda \int_{\partial\Omega}e^{2\lambda r^{-\beta}}   |\nabla \varphi|^2 d\sigma(\x)
	+C \int_{\Omega} e^{2\lambda r^{-\beta}} \big[
		 \lambda^3  |\varphi|^2
        + \lambda   |\nabla \varphi|^2
		 \big]d\x
	+ \lambda^2 
	\int_{\partial \Omega}  e^{2\lambda r^{-\beta}} |\partial \varphi|^2 d\sigma(\x)
	\\
		+\epsilon
	\|
		 \varphi		
	\|_{H}^2  
	\leq
	C\Big( \int_{\Omega}e^{2\lambda r^{-\beta}} |S \varphi|^2d\x
	+
	\int_{\Omega}e^{2\lambda r^{-\beta}} |{\bf F}(U) - {\bf F}(V)|^2d\x
	\Big).
	\label{4,7}
\end{multline}
Since $\varphi|_{\partial \Omega} = 0$, $|\nabla \varphi| = |\partial_{\nu} \varphi|$ on $\partial \Omega.$
Since $\lambda$ is large, the third integral on the left-hand side of equation \eqref{4,7} prevails over the first one. Moreover, the second integral on the left-hand side of equation \eqref{4,7} dominates the first integral on the right-hand side. Thus, we can deduce the following estimate
\begin{multline}
	 \int_{\Omega} e^{2\lambda r^{-\beta}} \big[
		 \lambda^3  |\varphi|^2
        + \lambda   |\nabla \varphi|^2
		 \big]d\x
	+ \lambda^2 
	\int_{\partial \Omega}  e^{2\lambda r^{-\beta}} |\partial \varphi|^2 d\sigma(\x)
		+\epsilon
	\|
		 \varphi		
	\|_{H}^2  
	\\
	\leq
	C
	\int_{\Omega}e^{2\lambda r^{-\beta}} |{\bf F}(U) - {\bf F}(V)|^2d\x.
	\label{4,8}
\end{multline}
Recalling \eqref{LipschitzF}, we have 
\begin{align*}
	 \lambda 
	 \Big(
	 	 \int_{\Omega} e^{2\lambda r^{-\beta}}& \big[
		 \lambda^2  |\varphi|^2
        +    |\nabla \varphi|^2
		 \big]d\x
		 +
	\lambda \int_{\partial \Omega}  e^{2\lambda r^{-\beta}} |\partial \varphi|^2 d\sigma(\x)
	+ \frac{\epsilon}{\lambda}
	\|
		 \varphi		
	\|_{H}^2  
	 \Big)		
	\\
	&\leq
	C
	\int_{\Omega}e^{2\lambda r^{-\beta}} |U - V|^2 + |\nabla (U - V)|^2d\x
	\\
	&
	\leq
	C\Big(
	\int_{\Omega}e^{2\lambda r^{-\beta}} (\lambda^2|U - V|^2 + |\nabla (U - V)|^2)d\x
	+
	\lambda \int_{\partial \Omega}  e^{2\lambda r^{-\beta}} |\partial (U - V)|^2 d\sigma(\x)
	\\
	&\hspace{11cm}+ \frac{\epsilon}{\lambda} \|U - V\|_H^2
	\Big).
\end{align*}
The desired estimate \eqref{4,1} follows.
 \end{proof}

\begin{remark}
 The proof of Theorem \ref{thm_contract} is similar to that of Theorem 3.1 in \cite{Nguyen:AVM2023}. However, the key distinction lies in the inclusion of the boundary integral within the norm $\|\cdot\|_{\lambda, \beta, \epsilon}$, resulting the convergence with respect to a stronger norm. This modification allows for the investigation of noise analysis without the need to impose a technical condition that the noise is the restriction of a smooth function, see \cite[Section 4]{Nguyen:AVM2023}. 
\end{remark}

 Define the sequence
	\begin{equation}
	\left\{
		\begin{array}{ll}
		U_0 \in H \mbox{ be chosen arbitrarily},\\
		U_{n + 1} = \Phi_{\lambda, \beta, \epsilon}(U_n)
		& n \geq 0
		\end{array}
	\right.
	\label{3.9}
	\end{equation}
A direct consequence of Theorem \ref{thm_contract} is that the sequence $\{U_n\}_{n \geq 1}$ converges to a vector-valued function $\overline U$ in H with respect to the norm $\|\cdot\|_{\lambda, \beta, \epsilon}.$
In the next section, we will show that $\overline U$ is a good approximation of the solution to \eqref{2.8}.

\section{The convergence of the Carleman contraction principle}\label{sec4}

We assume that the observed data $h = \partial_\nu u(\x, t)$, $(\x, t) \in \partial \Omega \times (0, T)$ contains noise.
As a result, the indirect data for \eqref{2.8}, the vector ${\bf h}$ defined in \eqref{indirectdata} is not accurate.
Denote by ${\bf h}^*$ the exact version of the vector  ${\bf h}$. 
Assume that problem \eqref{2.8} with ${\bf h}$ being replaced with ${\bf h}^*$ has a unique solution, denoted by $U^*$.

We have the theorem.

\begin{Theorem}
	Assume that \eqref{Lipschitz} and hence \eqref{LipschitzF} hold true. 
	Fix $\beta > \beta_0$ and $\lambda \geq \lambda_0$ sufficiently large where $\beta_0$ and $\lambda_0$ are as in Theorem \ref{thm_contract} such that $\Phi_{\lambda, \beta, \epsilon}$ is a contraction mapping for all $\epsilon > 0.$
	Let $\overline U$ be the fixed point of $\Phi_{\lambda, \beta, \epsilon}$. 
	We have
	\begin{multline}
\int_{\Omega}e^{2\lambda r^{-\beta}}  \big[
		 \lambda^3 |\overline U - U^*|^2
        + \lambda   |\nabla (\overline U - U^*)|^2
		 \big]d\x
	+ \lambda^2 \int_{\partial \Omega} e^{2\lambda r^{-\beta}} |\nabla (\overline U - U^*)|^2d\sigma(\x)
	\\
	+ \epsilon \|\overline U - U^*\|_H^2
	\leq 
	C\Big[
	\lambda^2 \int_{\partial \Omega}e^{2\lambda r^{-\beta}} |{\bf h} - {\bf h}^*|^2d\sigma(\x)
	+\epsilon
	\|
		  U^*		
	\|^2_H\Big].
	\label{4.21}
\end{multline}
where $C = C({\bf M}, \beta, \x_0, d, \Omega)$ is a constant. 
	\label{thm 4.1}
\end{Theorem}

\begin{proof}[Proof of Theorem \ref{thm 4.1}]
	It is well-known that the fixed point $\overline U$ of $\Phi_{\lambda, \beta, \epsilon}$ is the limit of the sequence $\{U_n\}_{n \geq 0}$, defined in \eqref{3.9}, with respect to the norm $\|\cdot\|_{\lambda, \beta, \epsilon}$
	Fix $n \geq 1.$
	Since $U_{n + 1} = \Phi_{\lambda, \beta, \epsilon}(U_n)$ is the minimizer of $J^{U_n}_{\lambda, \beta, \epsilon}$ in $H$, 	
by the variational principle, for all $\varphi \in H$, we have
\begin{multline}
	\langle
		e^{2\lambda r^{-\beta}} (\Delta U_{n + 1} - S U_{n + 1} + {\bf F}(U_n)),
		\Delta \varphi - S \varphi
	\rangle_{L^2(\Omega)^N}
		\\
	+ \lambda^2 
	\langle
		e^{2\lambda r^{-\beta}} (\partial_{\nu}U_{n + 1} - {\bf h}), \partial_{\nu} \varphi		
	\rangle_{L^2(\partial\Omega)^N}
	+\epsilon
	\langle
		 U_{n + 1}, \varphi		
	\rangle_{H} = 0
	\label{4.9}
\end{multline}
On the other hand, since $U^*$ is the solution to \eqref{2.8} with ${\bf h}^*$ replacing ${\bf h}$, for all $\varphi \in H^p(\Omega)^N$, we have
\begin{multline}
	\langle
		e^{2\lambda r^{-\beta}} (\Delta U^* - S U^* + {\bf F}(U^*)),
		\Delta \varphi - S \varphi
	\rangle_{L^2(\Omega)^N}
	\\
	+\lambda^2 
	\langle
		e^{2\lambda r^{-\beta}} (\partial_{\nu}U^* - {\bf h}^*), \partial_{\nu} \varphi		
	\rangle_{L^2(\partial\Omega)^N}
	+	\epsilon
	\langle
		 U^*, \varphi		
	\rangle_{H} 
	= 
	\epsilon
	\langle
		 U^*, \varphi		
	\rangle_{H}
	\label{4.10}
\end{multline}
for all $\varphi \in H.$
Subtracting \eqref{4.9} from \eqref{4.10}  and using
\begin{equation}
 	\varphi_{n + 1} = U_{n + 1} - U^* \in H
	\label{4.13}
\end{equation} 
as the test function $\varphi$ in \eqref{4.10}, we have
\begin{multline}
	\langle
		e^{2\lambda r^{-\beta}} (\Delta \varphi_{n + 1} - S \varphi_{n + 1} + {\bf F}(U_n) - {\bf F}(U^*)),
		\Delta \varphi_{n + 1} - S \varphi_{n + 1}
	\rangle_{L^2(\Omega)^N}
	\\
		+ \lambda^2
	\langle
		e^{2\lambda r^{-\beta}} \partial_{\nu}\varphi_{n + 1} , \partial_{\nu} \varphi_{n + 1}		
	\rangle_{L^2(\partial\Omega)^N}
	+\epsilon
	\langle
		 \varphi_{n + 1}, \varphi_{n + 1}		
	\rangle_{H}
	\\
	= 
	\lambda^2
	\langle
		e^{2\lambda r^{-\beta}} \partial_{\nu}\varphi_{n + 1}, \partial_{\nu} ({\bf h} - {\bf h}^*)		
	\rangle_{L^2(\partial\Omega)^N}
	-\epsilon
	\langle
		 U^*, \varphi_{n + 1}		
	\rangle_{H}.
	\label{4.12}
	\end{multline}
We can rewrite \eqref{4.12} as
\begin{multline}
	\int_{\Omega} e^{2\lambda r^{-\beta}}|\Delta \varphi_{n + 1} - S\varphi_{n + 1}|^2d\x 
	+ \lambda^2 \int_{\partial \Omega} e^{2\lambda r^{-\beta}} |\partial_{\nu}\varphi_{n + 1}|^2d\sigma(\x)
	+\epsilon \|\varphi_{n + 1}\|_H^2
	\\
	=
	-\int_{\Omega} e^{2\lambda r^{-\beta}}({\bf F}(U_{n}) - {\bf F}(U^*))(\Delta \varphi_{n + 1} - S\varphi_{n + 1}))^2d\x
	\\ 
	+ \lambda^2\int_{\partial \Omega} e^{2\lambda r^{-\beta}} \partial_{\nu} \varphi_{n + 1} \cdot ({\bf h} - {\bf h}^*) d\sigma(\x)
	-\epsilon
	\langle
		 U^*, \varphi_{n + 1}		
	\rangle_{H}.
	\label{4.14}
\end{multline}
Using the inequality $|ab| \leq \frac{1}{2}(a^2 + b^2),$ we deduce from \eqref{4.14} that
\begin{multline}
	\int_{\Omega} e^{2\lambda r^{-\beta}}|\Delta \varphi_{n + 1} - S\varphi_{n + 1}|^2d\x 
	+ \lambda^2 \int_{\partial \Omega} e^{2\lambda r^{-\beta}} |\partial_{\nu}\varphi_{n + 1}|^2d\sigma(\x)
	+\epsilon \|\varphi_{n + 1}\|_H^2
	\\
	\leq
	\int_{\Omega} e^{2\lambda r^{-\beta}}|{\bf F}(U_{n}) - {\bf F}(U^*)|^2d\x
	+ \lambda^2 \int_{\partial \Omega}e^{2\lambda r^{-\beta}} |{\bf h} - {\bf h}^*|^2d\sigma(\x)
	+\epsilon
	\|
		  U^*		
	\|^2_H.
	\label{4.15}
\end{multline}
Since $\varphi_{n + 1}|_{\partial \Omega} = 0$, $|\partial_{\nu} \varphi_{n + 1}| = |\nabla \varphi_{n + 1}|$ on $\partial \Omega$.
Also, due to the inequality $(a - b)^2 \geq \frac{1}{2} a^2 - b^2$,
we obtain from \eqref{4.15} that
\begin{multline}
	\frac{1}{2}\int_{\Omega} e^{2\lambda r^{-\beta}}|\Delta \varphi_{n + 1}|^2d\x - 
	\int_\Omega e^{2\lambda r^{-\beta}}|S\varphi_{n + 1}|^2d\x 
	+ \lambda^2 \int_{\partial \Omega} e^{2\lambda r^{-\beta}} |\nabla\varphi_{n + 1}|^2d\sigma(\x)
	+\epsilon \|\varphi_{n + 1}\|_H^2
	\\
	\leq
	\int_{\Omega} e^{2\lambda r^{-\beta}}|{\bf F}(U_{n}) - {\bf F}(U^*)|^2d\x
	+ \lambda^2 \int_{\partial \Omega}e^{2\lambda r^{-\beta}} |{\bf h} - {\bf h}^*|^2d\sigma(\x)
	+\epsilon
	\|
		 U^*		
	\|^2_H.
	\label{4.17}
\end{multline}
We now apply the Carleman estimate in \eqref{4.3333} to estimate the left-hand side of \eqref{4.17}. Combining \eqref{4.3333} for the vector-valued function $\varphi_{n + 1}$ and \eqref{4.17}, we obtain
\begin{multline}
	-C \lambda \int_{\partial\Omega}e^{2\lambda r^{-\beta}}   |\nabla \varphi_{n + 1}|^2 d\sigma(\x)
	+C \int_{\Omega}e^{2\lambda r^{-\beta}}\big[
		 \lambda^3  |\varphi_{n + 1}|^2
        + \lambda  |\nabla \varphi_{n + 1}|^2
		 \big]d\x - 
	\int_\Omega e^{2\lambda r^{-\beta}}|S\varphi_{n + 1}|^2d\x 
	\\
	+ \lambda^2 \int_{\partial \Omega} e^{2\lambda r^{-\beta}} |\nabla\varphi_{n + 1}|^2d\sigma(\x)
	+\epsilon \|\varphi_{n + 1}\|_H^2
	\leq
	\int_{\Omega} e^{2\lambda r^{-\beta}}|{\bf F}(U_{n}) - {\bf F}(U^*)|^2d\x
	\\
	+ \lambda^2 \int_{\partial \Omega}e^{2\lambda r^{-\beta}} |{\bf h} - {\bf h}^*|^2d\sigma(\x)
	+\epsilon
	\|
		 U^*		
	\|^2_H.
	\label{4.18}
\end{multline}
Letting $\lambda$ large, we can simplify \eqref{4.18} as
\begin{multline}
\int_{\Omega}e^{2\lambda r^{-\beta}}  \big[
		 \lambda^3 |\varphi_{n+1}|^2
        + \lambda   |\nabla \varphi_{n+1}|^2
		 \big]d\x
	+ \lambda^2 \int_{\partial \Omega} e^{2\lambda r^{-\beta}} |\nabla\varphi_{n+1}|^2d\sigma(\x)
	+ \epsilon \|\varphi_{n+1}\|_H^2
	\\
	\leq 
	C \Big[
	 \int_{\Omega} e^{2\lambda r^{-\beta}}|{\bf F}(U_{n}) - {\bf F}(U^*)|^2d\x
	 + \lambda^2 \int_{\partial \Omega}e^{2\lambda r^{-\beta}} |{\bf h} - {\bf h}^*|^2d\sigma(\x)
	+ \epsilon
	\|
		  U^*		
	\|^2_H\Big].
	\label{4.19}
\end{multline}
We now employ the Lipschitz continuity of ${\bf F}.$
It follows from \eqref{LipschitzF} and \eqref{4.19} that
\begin{multline}
\int_{\Omega}e^{2\lambda r^{-\beta}}  \big[
		 \lambda^3 |\varphi_{n+1}|^2
        + \lambda   |\nabla \varphi_{n+1}|^2
		 \big]d\x
	+ \lambda^2 \int_{\partial \Omega} e^{2\lambda r^{-\beta}} |\nabla\varphi_{n+1}|^2d\sigma(\x)
	+ \epsilon \|\varphi_{n+1}\|_H^2
	\\
	\leq 
	C\Big[{\bf M}\int_{\Omega} e^{2\lambda r^{-\beta}}[|U_{n} - U^*|^2 + |\nabla (U_n - U^*)|^2]d\x
	\\
	+ \lambda^2 \int_{\partial \Omega}e^{2\lambda r^{-\beta}} |{\bf h} - {\bf h}^*|^2d\sigma(\x)
	+\epsilon
	\|
		  U^*		
	\|^2_H\Big].
	\label{4.20}
\end{multline}
Recall from \eqref{4.13} that $\varphi_{n + 1} = U_{n + 1} - U^*$ and that $U_n \to \overline U$ in $H$ with respect to the norm $\|\cdot\|_{\lambda, \beta, \epsilon}$. When the parameters $\lambda$, $\beta$, and $\epsilon$ are fixed, we can conclude that $U_n \to \overline U$ with respects to all of the norms $L^2(\Omega)$, $H^1(\Omega)$, and $H^p(\Omega).$
Letting $n$ in \eqref{4.20} to $\infty$, we have
\begin{multline}
\int_{\Omega}e^{2\lambda r^{-\beta}}  \big[
		 \lambda^3 |\overline U - U^*|^2
        + \lambda   |\nabla (\overline U - U^*)|^2
		 \big]d\x
	+ \lambda^2 \int_{\partial \Omega} e^{2\lambda r^{-\beta}} |\nabla (\overline U - U^*)|^2d\sigma(\x)
	+ \epsilon \|\overline U - U^*\|_H^2
	\\
	\leq 
	C\Big[{\bf M}\int_{\Omega} e^{2\lambda r^{-\beta}}[|\overline U - U^*|^2 + |\nabla (\overline U - U^*)|^2]d\x
	\\
	+ \lambda^2 \int_{\partial \Omega}e^{2\lambda r^{-\beta}} |{\bf h} - {\bf h}^*|^2d\sigma(\x)
	+\epsilon
	\|
		  U^*		
	\|^2_H\Big].
	\label{4.2020}
\end{multline}
Letting $\lambda$ be sufficiently large, we can use the first integral in the left-hand side of \eqref{4.2020} to dominate the first integral in the right-hand side of \eqref{4.2020}.
We obtain \eqref{4.21}.
\end{proof}

Theorem \ref{thm 4.1} leads to Algorithm \ref{alg1} to solve the time-reduction model \eqref{2.8}.
\begin{algorithm}[h!]
\caption{\label{alg1} Computing Numerical Solutions to \eqref{2.8}}
	\begin{algorithmic}[1]
	\State \label{s1} Select a regularization parameter $\epsilon$ and a minimum threshold $\kappa_0 > 0$.
	 \State \label{s2}  Initialize with $n = 0$ and an initial solution $U_0 \in H$.
		\State \label{step update} 		
		Update $U_{n + 1} = \Phi_{\lambda, \beta, \epsilon}(U_n)$ by minimizing $J^{U_n}_{\lambda, \beta, \epsilon}$ in $H$.   
	\If {$\|U_{n + 1} - U_n\|_{L^2(\Omega)} > \kappa_0$}
		\State Increment $n$ to $n + 1$.
		\State Return to Step \ref{step update}.
	\Else	
		\State Finalize the solution as $U_{\rm comp} = U_{n + 1}$.
	\EndIf
\end{algorithmic}
\end{algorithm}

\begin{remark}
In the statement of Theorem \ref{thm 4.1}, we imposed a technical condition regarding the Lipschitz continuity of the nonlinear and nonlocal operator $F$ in \eqref{Lipschitz}, and consequently, the nonlinear function ${\bf F}$ in \eqref{LipschitzF}. One might thus assume that the Lipschitz condition is a necessary requirement for Algorithm \ref{alg1} and Algorithm \ref{alg3}. However, this assumption can be relaxed under certain circumstances. 
Suppose we possess an upper bound on $U^*$, say $\|U^*\|_{C^1(\overline \Omega)} \leq M$, where $M$ is a positive constant. In such cases, we only need to compute the solution within the set $\{U \in H: \|U\|_{C^1(\overline \Omega)} \leq M\}$. To facilitate this, we define the following functions
 \[
 	\chi_M(\x, s, {\bf p}) = \left\{
		\begin{array}{ll}
			1 & s^2 + |{\bf p}|^2 \leq M,\\
			\in (0, 1) &M < s^2 + |{\bf p}|^2  \leq 2M,\\
			0 &s^2 + |{\bf p}|^2 > 2M,
		\end{array}
	\right.
	\quad
	\mbox{and}
	\quad
	{\bf F}_M = \chi_M \bf{F}.
 \]
It is obvious that $U^*$ satisfies
\begin{equation*}
	\left\{
		\begin{array}{ll}
			\Delta U(\x) - SU(\x) + {\bf F}_MU(\x) = 0 &\x \in \Omega,\\
			U(\x) = 0 &\x \in \partial \Omega,\\
			\partial_{\nu} U(\x) = {\bf h}(\x) & \x \in \partial \Omega.
		\end{array}
	\right.
\end{equation*}
We can then compute $U^*$ using Algorithm \ref{alg1}, replacing ${\bf F}$ with ${\bf F}_M$. In cases where we do not possess an upper bound for $\|U^*\|_{C^1(\overline \Omega)}$, we can apply the aforementioned procedure for some value of $M$ to compute $U_{\text{comp}}^M$, and subsequently, let $M \to \infty$. The convergence of the computed solution is guaranteed as long as the true solution $U^*$ lies within the class $C^1$.
\end{remark}

\section{Numerical study} \label{num}

In this section, we present some numerical examples.
The first step to generate the noisy simulated data, see the following subsection.

\subsection{Data generation}
We consider the case when $d = 2$.
Set $\Omega = (-R, R)^d$ where $R = 1$.
Fix a number $N_{\x}$. We arrange an $N_\x \times N_\x$ uniform partition of $\overline \Omega$ with the grid points 
\[
	\mathcal G = \big\{ \x_{ij} = 
	(-R + (i-1)\delta_{\x},  -R + (j - 1)\delta_{\x}): 1 \leq i, j \leq N_{\x}\big\} \subset \overline \Omega
\]
where $\delta_\x = \frac{2R}{N_\x - 1}$. Here, $N_\x = 81.$
We also discretize the time domain $[0, T]$ by $N_T$ points 
\[
	\mathcal T = \{t_l = (l - 1)\delta_t: 1 \leq l \leq N_T\}
\]
where $\delta_t = \frac{T}{N_T - 1}$ for $N_T = 200$ and $T = 2$.
To solve the forward problem, we employ a combination of explicit and implicit schemes for $u(x_i, y_j , t_l)$, where $1 \leq i, j \leq N_\x$ and $1 \leq l \leq N_T$. The algorithm used to implement this approach is described in Algorithm \ref{alg2}.
\begin{algorithm}[h!]
\caption{\label{alg2}The procedure to generate the data}
	\begin{algorithmic}[1]
	\State Set $u(\x_{ij}, t_1) = u(\x_{ij}, t_2) =  g(\x_{ij})$ for $\x_{ij} \in \mathcal G$.
	 \For{$l = 3$ to $N_T$}
	 	\State \label{linear} Solve the boundary value linear elliptic problem
		\begin{equation}
		\left\{
			\begin{array}{ll}
			\frac{w - 2u(\x, t_{l-1}) + u(\x, t_{l - 2})}{\delta_t^2} = \Delta w + \mathcal F(u(\x, t_{l-1})) &\x \in \mathcal{G},\\
			w(\x) = 0 &\x \in \mathcal G \cap \partial \Omega.
			\end{array}
		\right.
		\label{discretize}
		\end{equation}
		for a function $w$.
		\State Set $u(\x, t_{l}) = w(\x)$ for all $\x \in \mathcal G.$
	 \EndFor
	 \State \label{rand}Compute the noisy data
	 \[
	 	h(\x, t_l) = \partial_{\nu} u(\x, t_{l}) (1 + \delta \mbox{rand})
	\] on $\mathcal G \times \mathcal T$ where $\delta$ is the noise level.
\end{algorithmic}
\end{algorithm}
Solving for $w$ in Step \ref{linear} of Algorithm \ref{alg2} is standard since \eqref{discretize} is linear with respect to $w$. One can download a package to solve elliptic PDE with given Dirichlet boundary data on https://github.com/nhlocnguyenIP/Elliptic to solve \eqref{discretize}. In Step \ref{rand} of Algorithm \ref{alg2}, the function ``rand" gives a uniformly distributed random number in the range $[-1, 1].$
In all of our numerical tests, the noise level $\delta$ is $10\%.$

\subsection{The implementation for the inverse problem}

\begin{algorithm}[h!]
\caption{\label{alg3}The procedure solve Problem \ref{isp}}
	\begin{algorithmic}[1]
		\State \label{N} Choose a cut-off number $N$.
		\State \label{SolvePDEs} Use Algorithm \ref{alg1} to compute a solution $U_{\rm comp}$ to \eqref{2.8}.
		\State \label{comp_g} Write $U_{\rm comp} = (u_1^{\rm comp}, \dots, u_N^{\rm comp}).$ Compute the source function $g$ by the following formula
		\[
			g_{\rm comp}(\x) = \sum_{n = 1}^N u_n^{\rm comp}(\x) \Psi_n(0)
		\]
		for all $\x \in \Omega.$
	\end{algorithmic}
\end{algorithm}

The method for calculating the source function $g$ is detailed in Algorithm \ref{alg3}. It is important to note that the initial condition $u_t(\x, 0) = 0$, $\x \in \Omega$, is not explicitly used in the derivation of Algorithm \ref{alg3}. Nevertheless, the uniqueness of Problem \ref{isp} cannot be assured without the knowledge of $u(\x , 0)$, $\x \in \Omega$. This implies that the solution obtained might not be the desired one, indicating that this condition is implicitly incorporated in our approach.

We interpret our choice of $N$ in Step \ref{N}. 
This choice depends on how good the approximation of the function $u$ and its approximation in \eqref{2.1}. We, however, the internal information of $u$ is unavailable.
We only test the approximation on $\partial \Omega$.
Take the data $h$ at a point $\x^* \in \partial \Omega$.
For each $N$, define
\[
	e_N(t) = \Big|h(\x^*, t) - \sum_{n = 1}^N h_n(\x^*) \Psi_N(t)\Big|
\]
where $h_n(x^*) = \ds\int_0^T h(\x^*, t) \Psi_n(t)dt.$
A number $N$ is chosen if $\|e_N\|_{L^\infty([0, T])} < \varepsilon$ for some $\varepsilon \ll 1.$
In our computation, we choose $\varepsilon = 5\times 10^{-3}$ and $N = 40$.
Figure \ref{fig_choiceN} illustrates this procedure when the data $h$ is taken from Test 1 below.
\begin{figure}[h!]
	\subfloat[$N = 20$, the functions $h(\x^*, t)$ (blue, solid) and its approximation  $\sum_{n = 1}^N h_n(\x^*) \Psi_N(t).$]{\includegraphics[width=.3\textwidth]{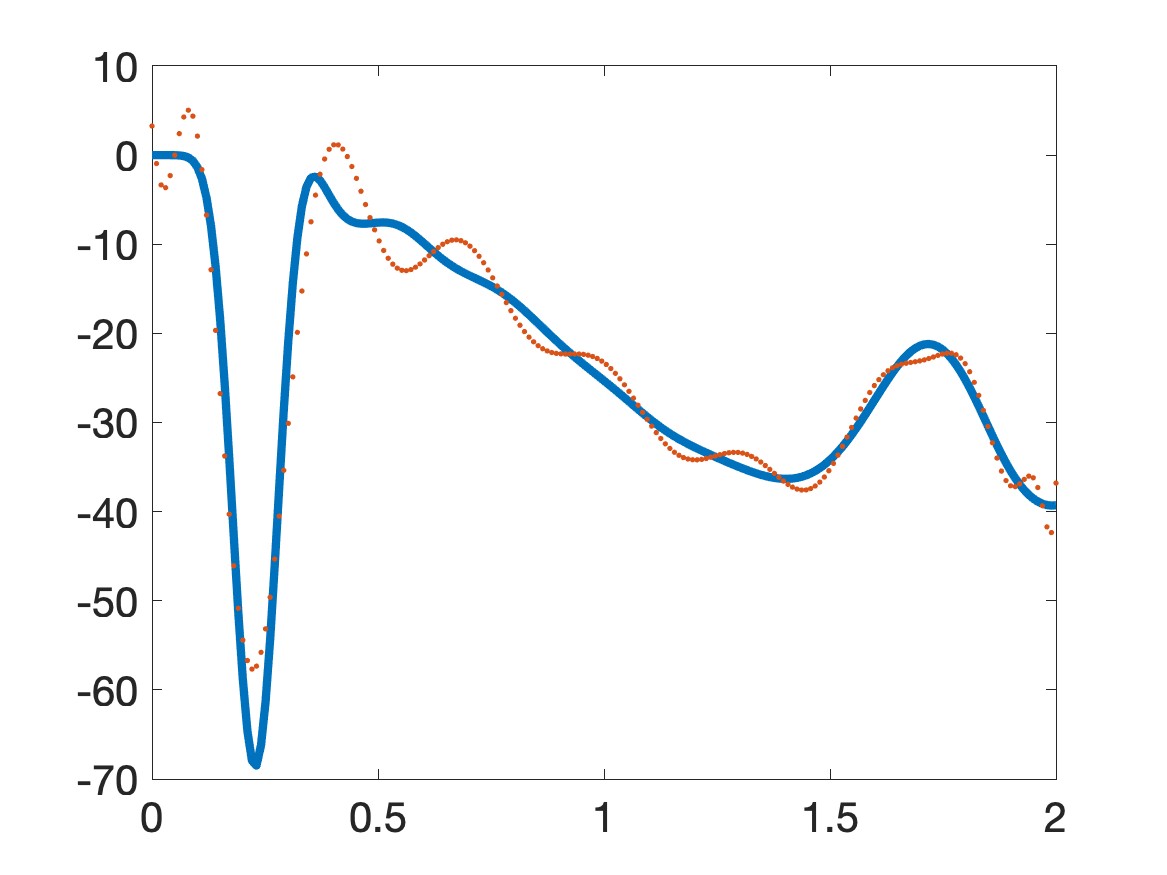}}
	\quad 
	\subfloat[$N = 30$, the functions $h(\x^*, t)$ (blue, solid) and its approximation  $\sum_{n = 1}^N h_n(\x^*) \Psi_N(t).$]{\includegraphics[width=.3\textwidth]{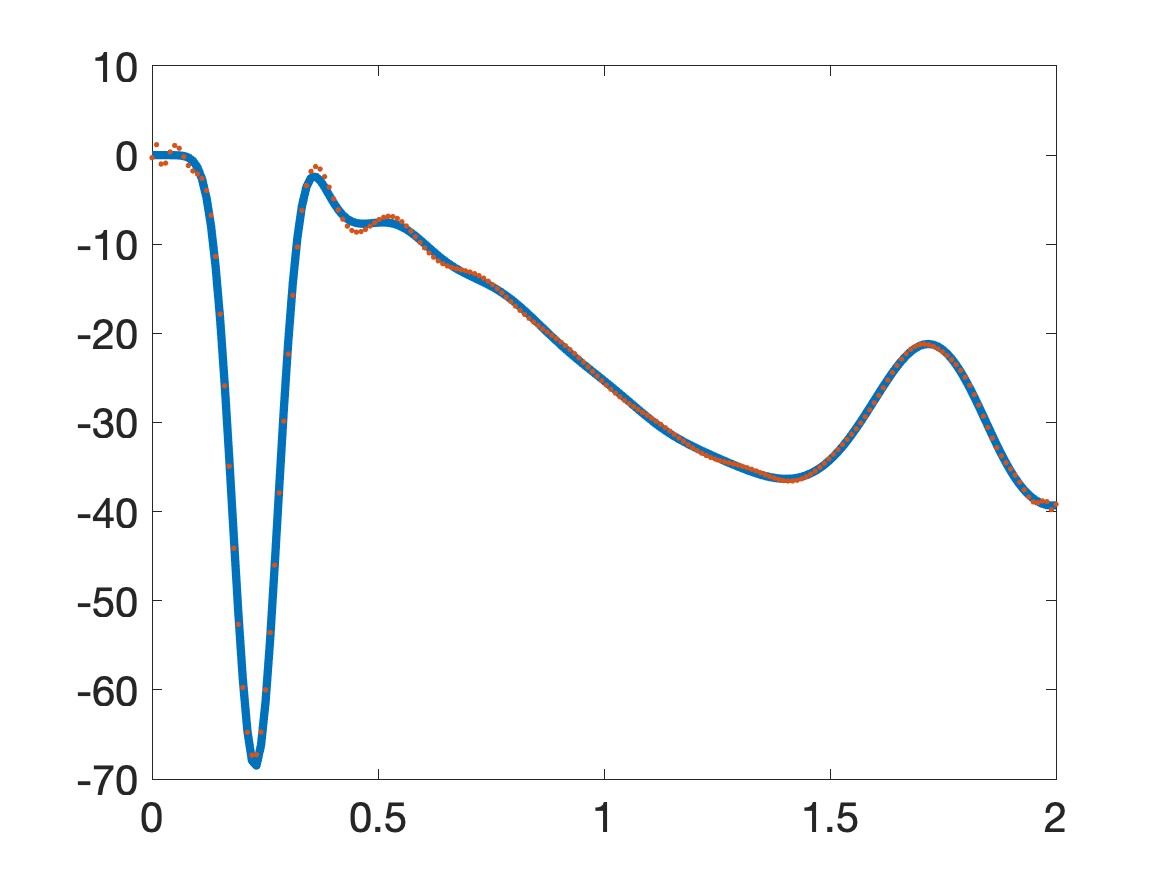}}
	\quad 
	\subfloat[$N = 40$, the functions $h(\x^*, t)$ (blue, solid) and its approximation  $\sum_{n = 1}^N h_n(\x^*) \Psi_N(t).$]{\includegraphics[width=.3\textwidth]{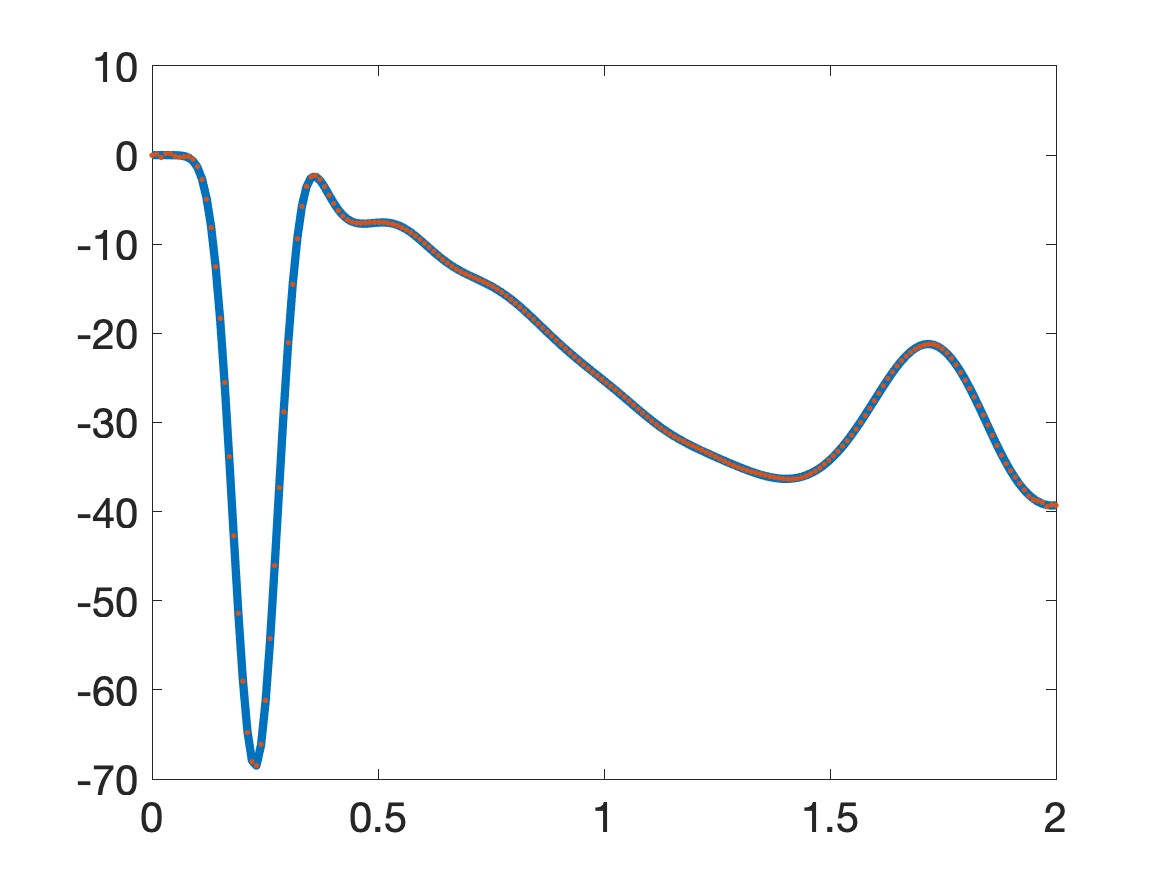}}
	
	\subfloat[$N = 20$, the functions $e_N(t)$.]{\includegraphics[width=.3\textwidth]{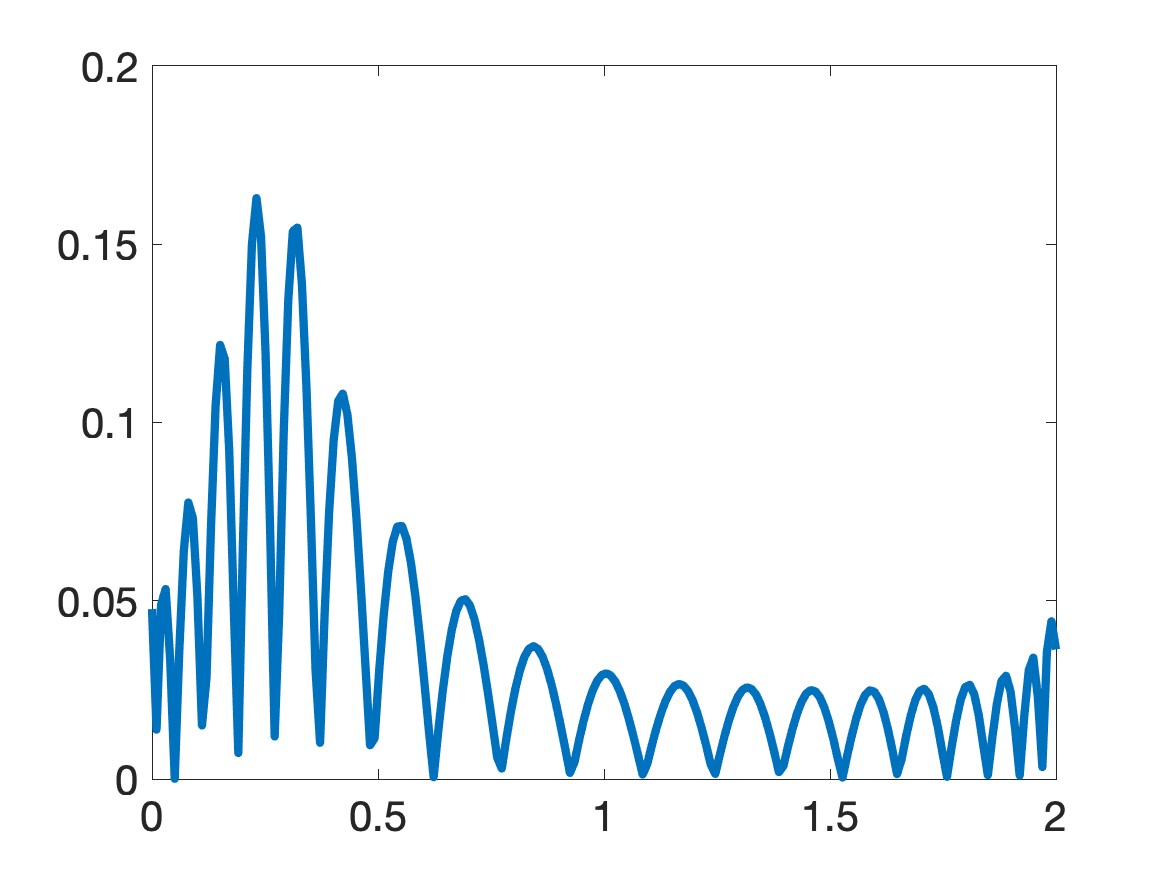}}
	\quad
	\subfloat[$N = 30$, the functions $e_N(t)$.]{\includegraphics[width=.3\textwidth]{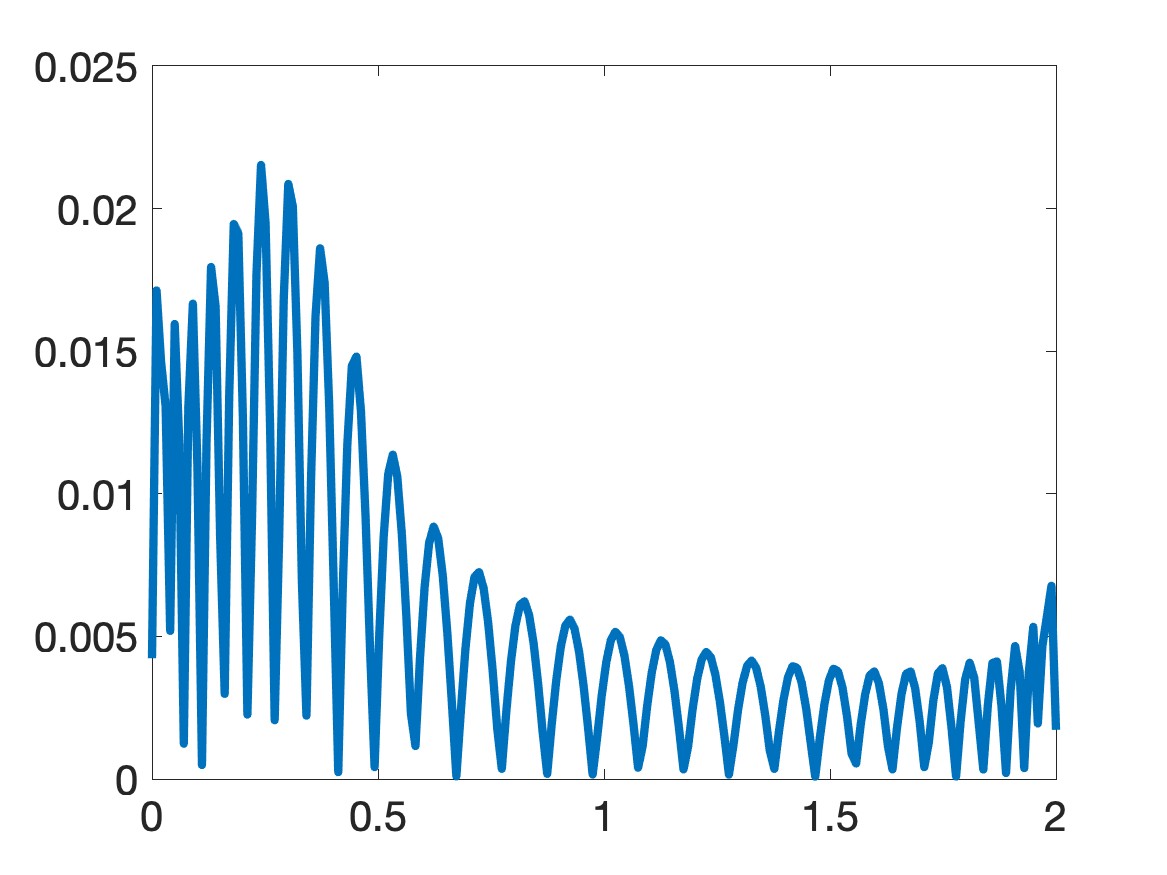}}
	\quad
	\subfloat[$N = 40$, the functions $e_N(t)$.]{\includegraphics[width=.3\textwidth]{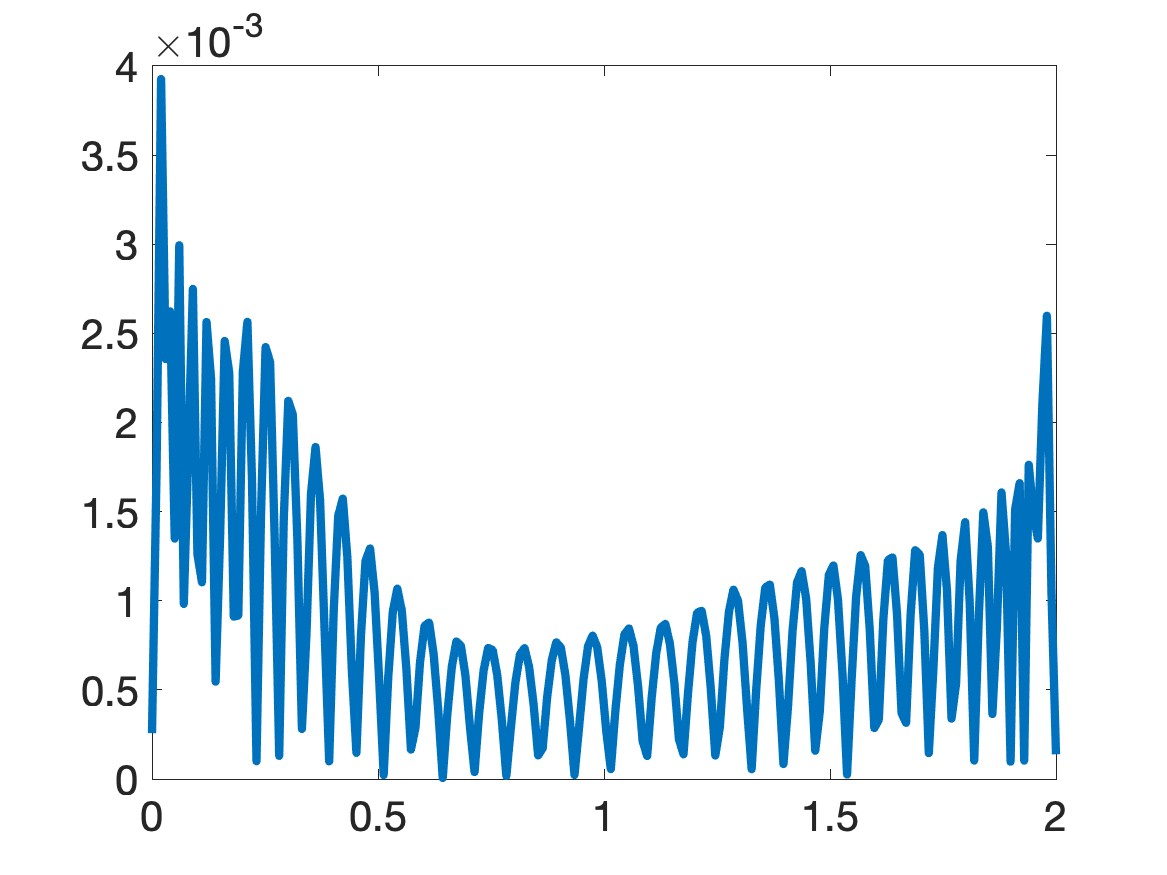}}
	\caption{\label{fig_choiceN} It is evident that when $N = 40$, the data $h(\x^*, \cdot)$ is well approximated by cutting its Fourier series. The function $h$ in these figures is the data for Test 1 in this section. The point $\x^* = (-1, 0).$}
\end{figure}

We refer the reader to \cite{Nguyen:AVM2023} for the implementation of Step \ref{SolvePDEs}.
In this step, the parameters $\epsilon = 10^{-13}$, $\lambda = 6$ and $\beta = 10$ are chosen by a trial and error process.
We take a reference test (test 1), which we assume to know the true solution. We then run Algorithm \ref{alg3} with many values of these parameters until we obtain acceptable solutions. 
We then use these parameters for all other tests.
 Step \ref{comp_g} is straightforward.

\subsection{Numerical examples}

We provide three (3) tests.

{\bf Test 1.} We consider the case when
\[
	\mathcal{F}(u) = \min\Big\{u^2 + |\nabla u|, 30\Big\} + \int_0^t u(\x, s) ds
\] and the true source function is given by
\[
	g_{\rm true}(\x) = 
	\left\{
		\begin{array}{ll}
			10 & \mbox{if }  x^2 + 3y^2 < 0.8^2,\\
			0 &\mbox{otherwise}.
		\end{array}
	\right.
\]
The solutions of this test are displayed in Figure \ref{fig_test1}.
\begin{figure}[h!]
\begin{center}
	\subfloat[The true function $g_{\rm true}$]{\includegraphics[width=.3\textwidth]{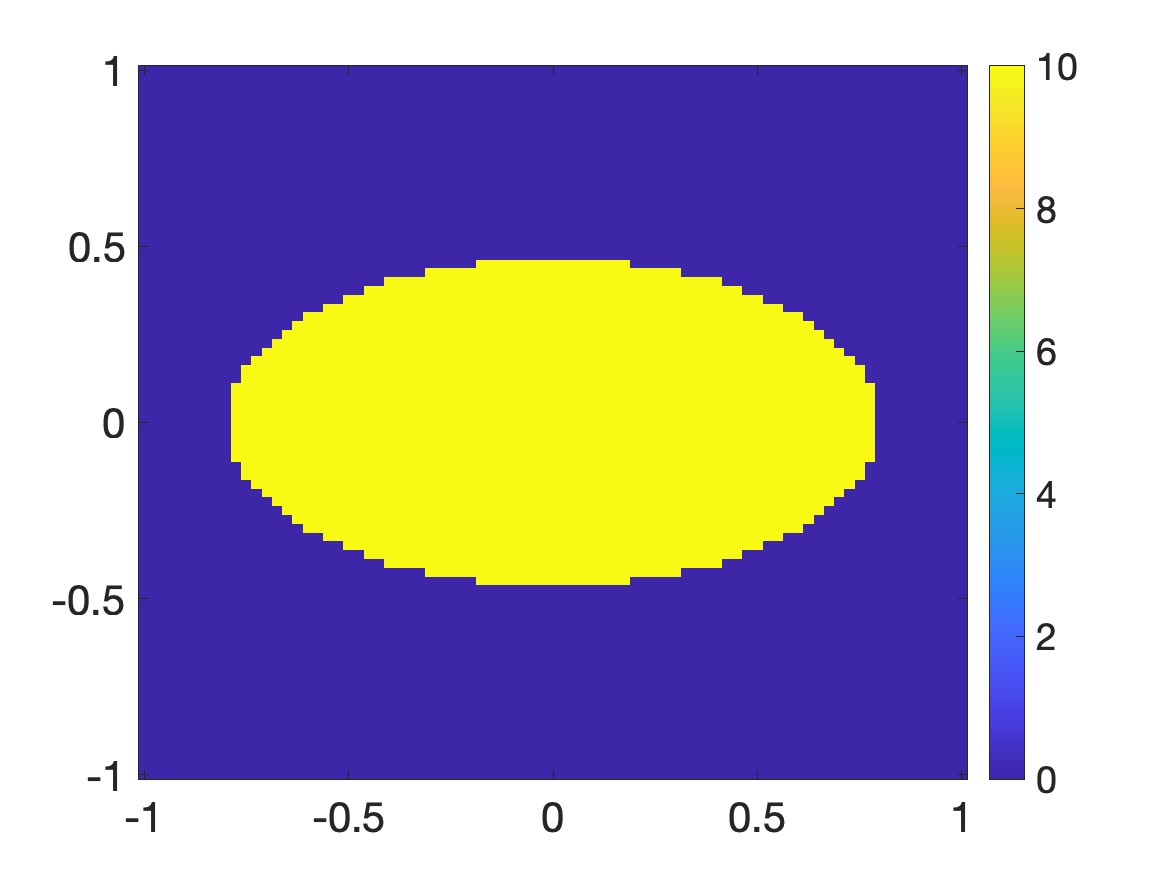}}
	\quad
	\subfloat[The computed function $g_{\rm comp}$]{\includegraphics[width=.3\textwidth]{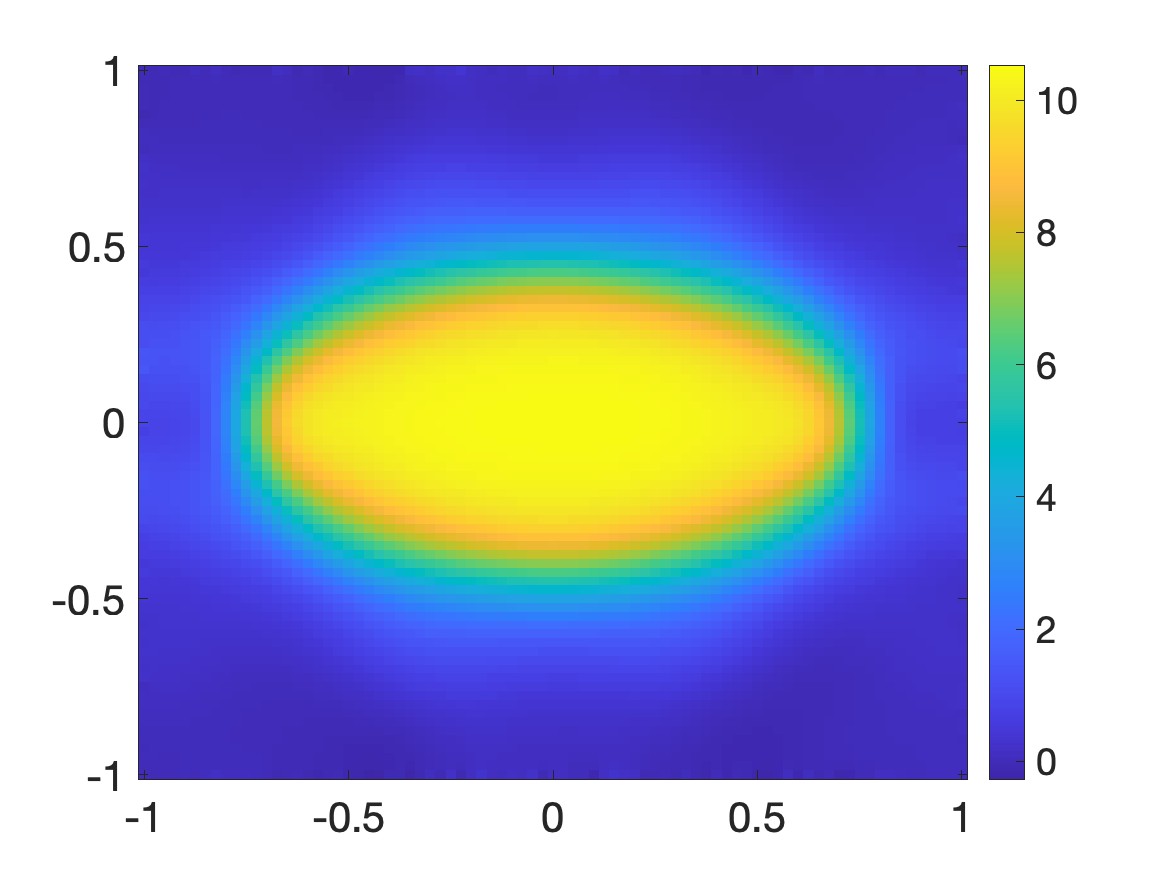}}
	\quad
	\subfloat[The difference $\frac{|g_{\rm comp} - g_{\rm true}|}{\|g_{\rm true}\|_{L^{\infty}}}$]{\includegraphics[width=.3\textwidth]{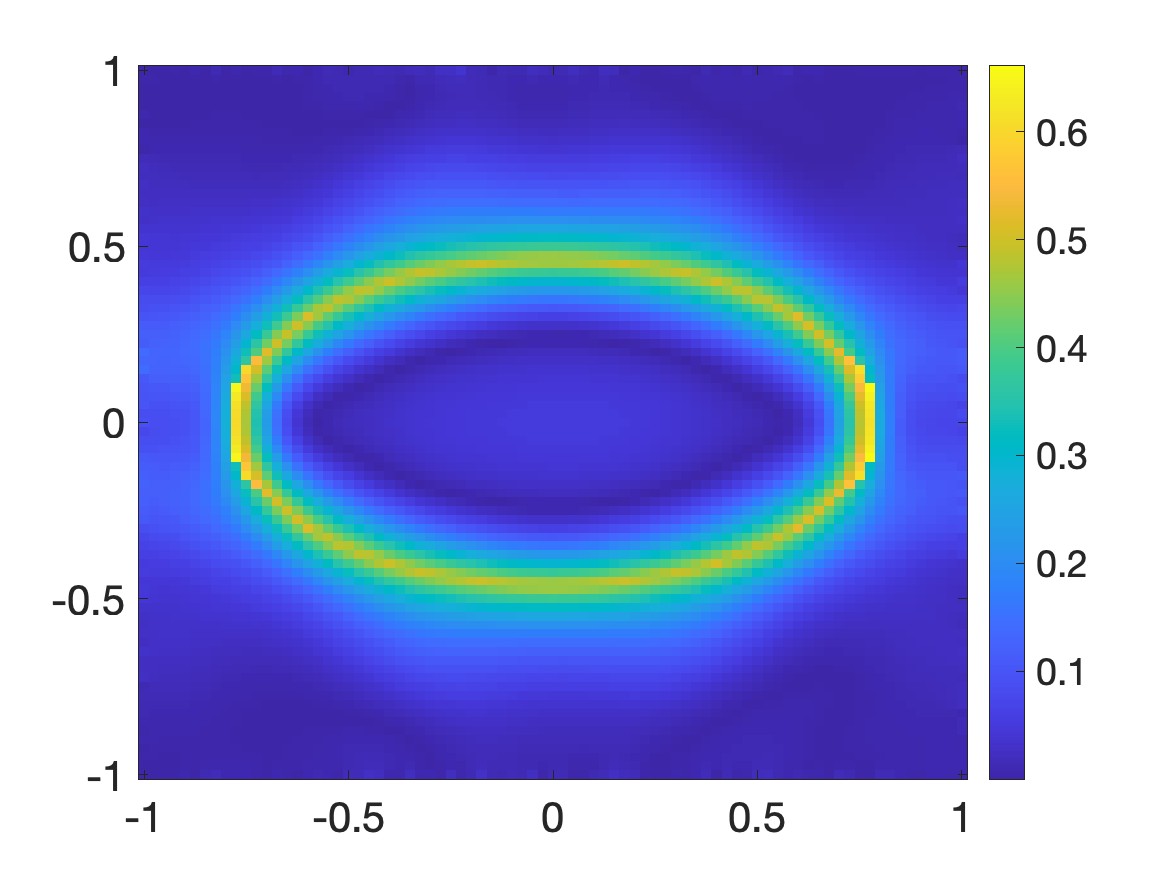}}
	\caption{\label{fig_test1} True and numerical solutions of test 1. 
	It is interesting mentioning that although the true solution has a high value (10) and the size of the ``ellipse inclusion" is not small, our method can deliver a satisfactory solution without requesting a good initial guess. The error in computation occurs mostly at the boundary of the inclusion.
	}
	\end{center}
\end{figure}

This test is challenging since the nonlinearity is not smooth. The growth of $u$ is of the quadratic function.
It is also interesting with the presence of the nonlocal term.
However, it is evident that our method provides a good numerical solution, although this test is complicated. 
The maximum value of the reconstructed source function $g_{\rm comp}$ in the ellipse inclusion is 10.5148 (the relative error is 5.15\%).

{\bf Test 2.} We test the case when 
\[
	\mathcal{F}(u) =  \frac{1}{\sqrt{u^2 + |\nabla u|^2}} + \int_0^t \frac{u(\x, s)}{1 + s^2} ds
\] and the true source function is given by
\[
	g_{\rm true}(\x) = 
	\left\{
		\begin{array}{ll}
			5 & \mbox{if }  \max\{|x - 0.5|/0.35, |y|/0.8\} < 1,\\
			4 & \mbox{if }  (x + 0.5)^2 + y^2 > 0.35^2,\\
			0 &\mbox{otherwise}.
		\end{array}
	\right.
\]
The support of the true source function consists of two ``inclusions", one rectangle and one disk.
The solutions of this test are displayed in Figure \ref{fig_test2}.
\begin{figure}[h!]
\begin{center}
	\subfloat[The true function $g_{\rm true}$]{\includegraphics[width=.3\textwidth]{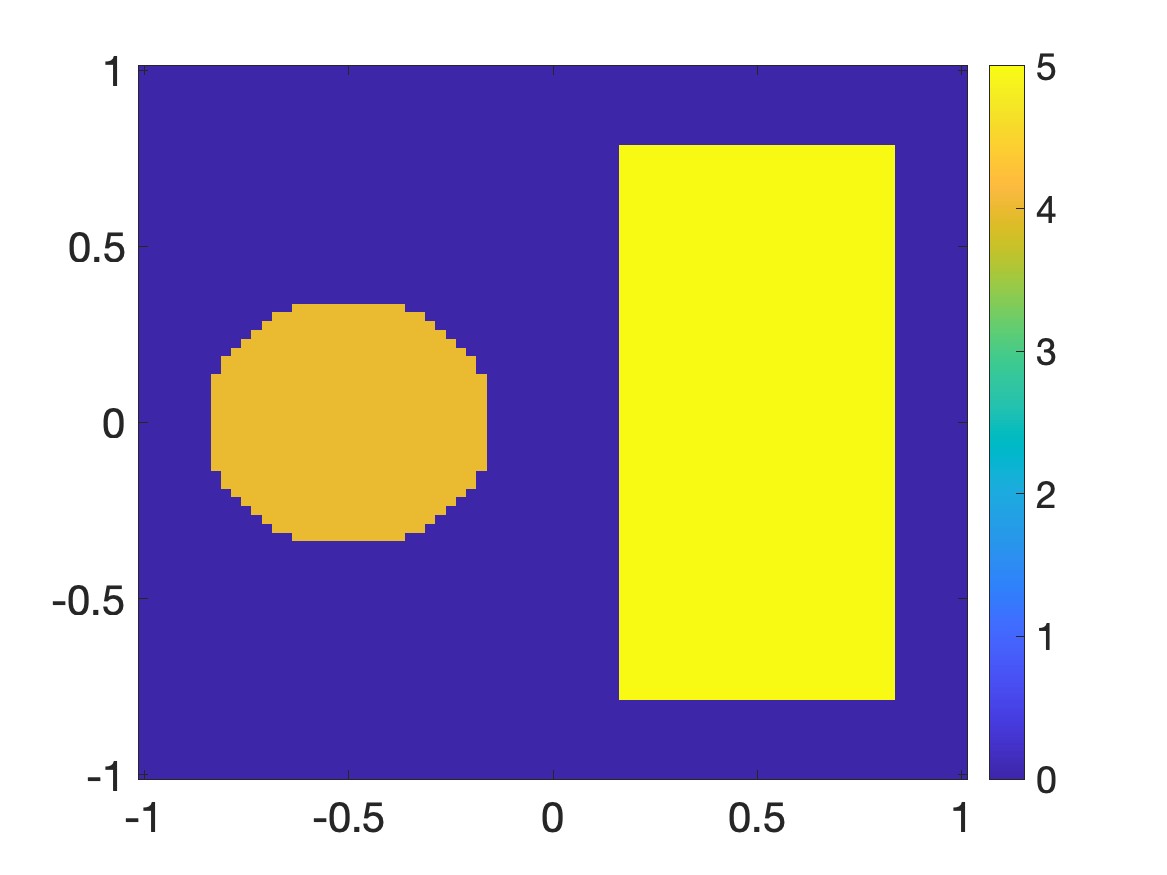}}
	\quad
	\subfloat[The computed function $g_{\rm comp}$]{\includegraphics[width=.3\textwidth]{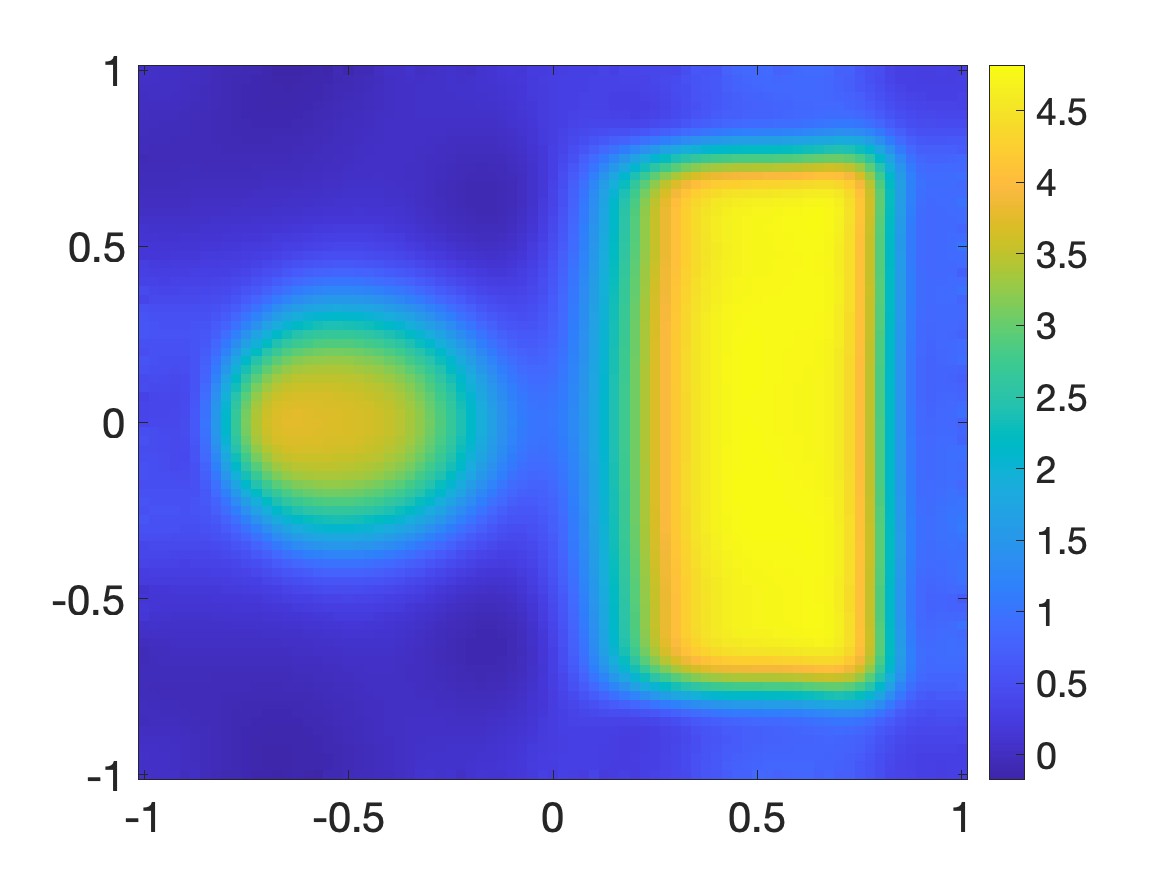}}
	\quad
	\subfloat[The difference $\frac{|g_{\rm comp} - g_{\rm true}|}{\|g_{\rm true}\|_{L^{\infty}}}$]{\includegraphics[width=.3\textwidth]{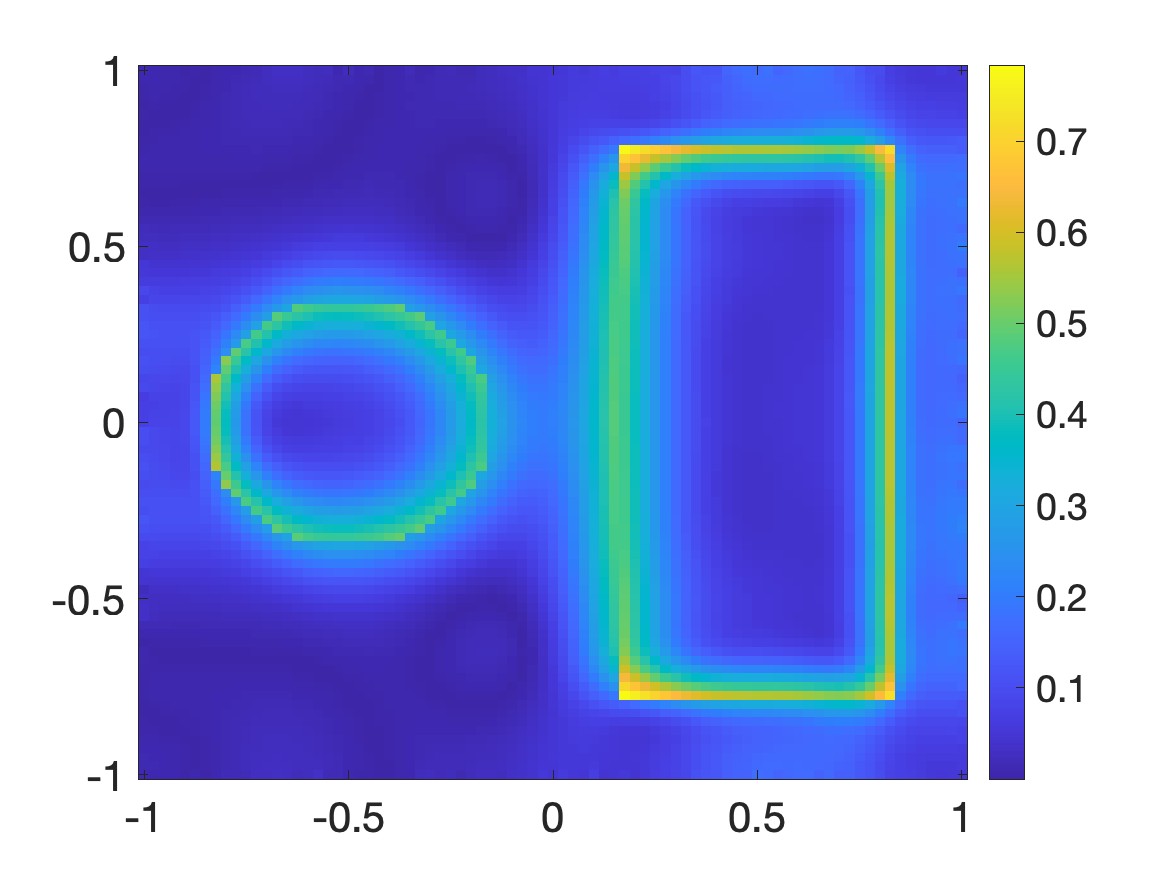}}
	\caption{\label{fig_test2} True and numerical solutions of test 2. 
	Like in test 1, our method can deliver a satisfactory solution for test 2 without requesting a good initial guess. Also, the error in computation occurs mostly at the boundary of the inclusion.
	}
	\end{center}
\end{figure}

It is evident that our method provides a good numerical solution, although the true source function has two inclusions, in each of which, the source function takes different values. 
These values are high (5 and 4).
The maximum value of the reconstructed source function $g_{\rm comp}$ in the rectangular inclusion is 4.8 (the relative error is 3.74\%).
The maximum value of the reconstructed source function $g_{\rm comp}$ in the circular inclusion is 3.75 (the relative error is 6.25\%).

{\bf Test 3.} We test the case when 
\[
	\mathcal{F}(u) =  u\ln(u^2 + 1) + u_x + u_y + \int_0^t u(\x, s) ds
\] and the true source function is given by
\[
	g_{\rm true}(\x) = 
	\left\{
		\begin{array}{ll}
			7 & \mbox{if }  \max\{|x + 0.6|/0.25, |y - 0.2|/0.7\} < 1,\\
			7 & \mbox{if } \max\{|x + 0.5|/0.25, |y|/0.7\} < 1, \\
			0 &\mbox{otherwise}.
		\end{array}
	\right.
\]
The support of the true source function is an $L$ shape conclusion.
The solutions of this test are displayed in Figure \ref{fig_test3}.
\begin{figure}[h!]
\begin{center}
	\subfloat[The true function $g_{\rm true}$]{\includegraphics[width=.3\textwidth]{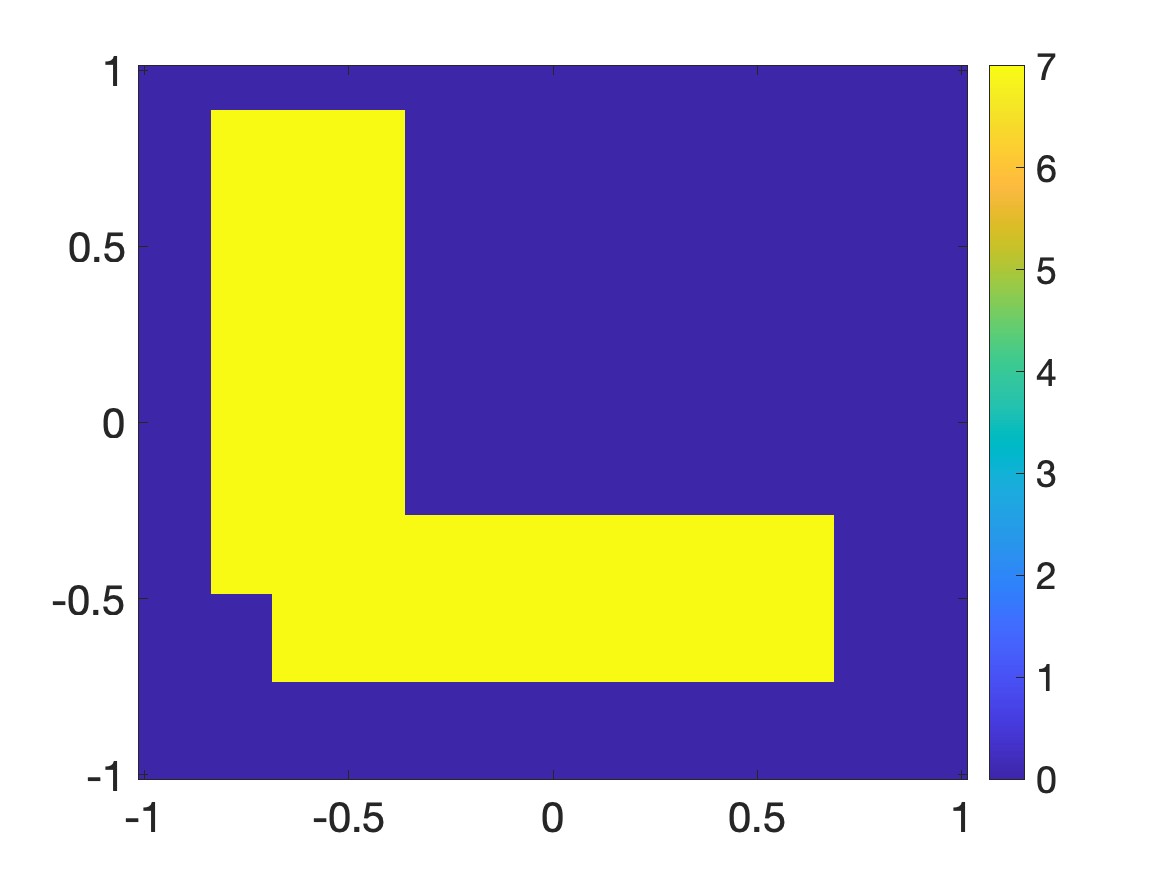}}
	\quad
	\subfloat[The computed function $g_{\rm comp}$]{\includegraphics[width=.3\textwidth]{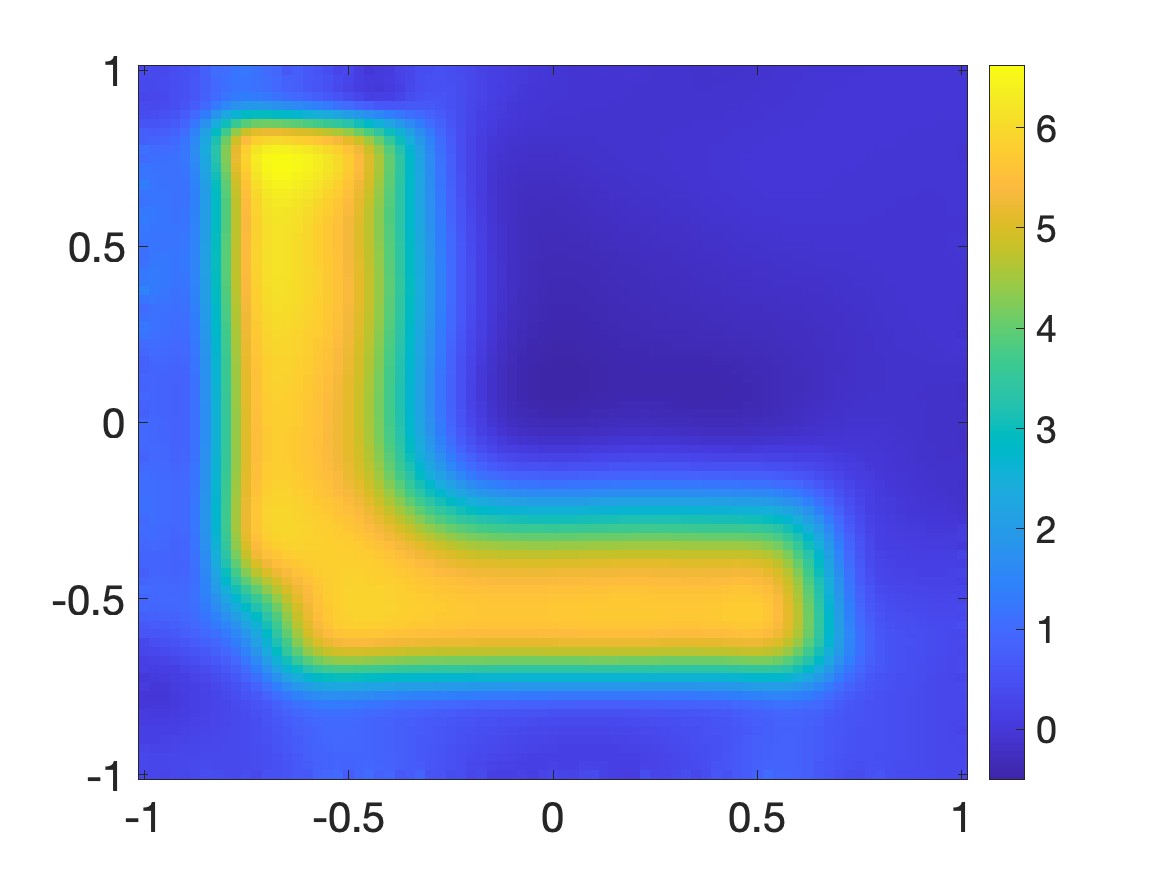}}
	\quad
	\subfloat[The difference $\frac{|g_{\rm comp} - g_{\rm true}|}{\|g_{\rm true}\|_{L^{\infty}}}$]{\includegraphics[width=.3\textwidth]{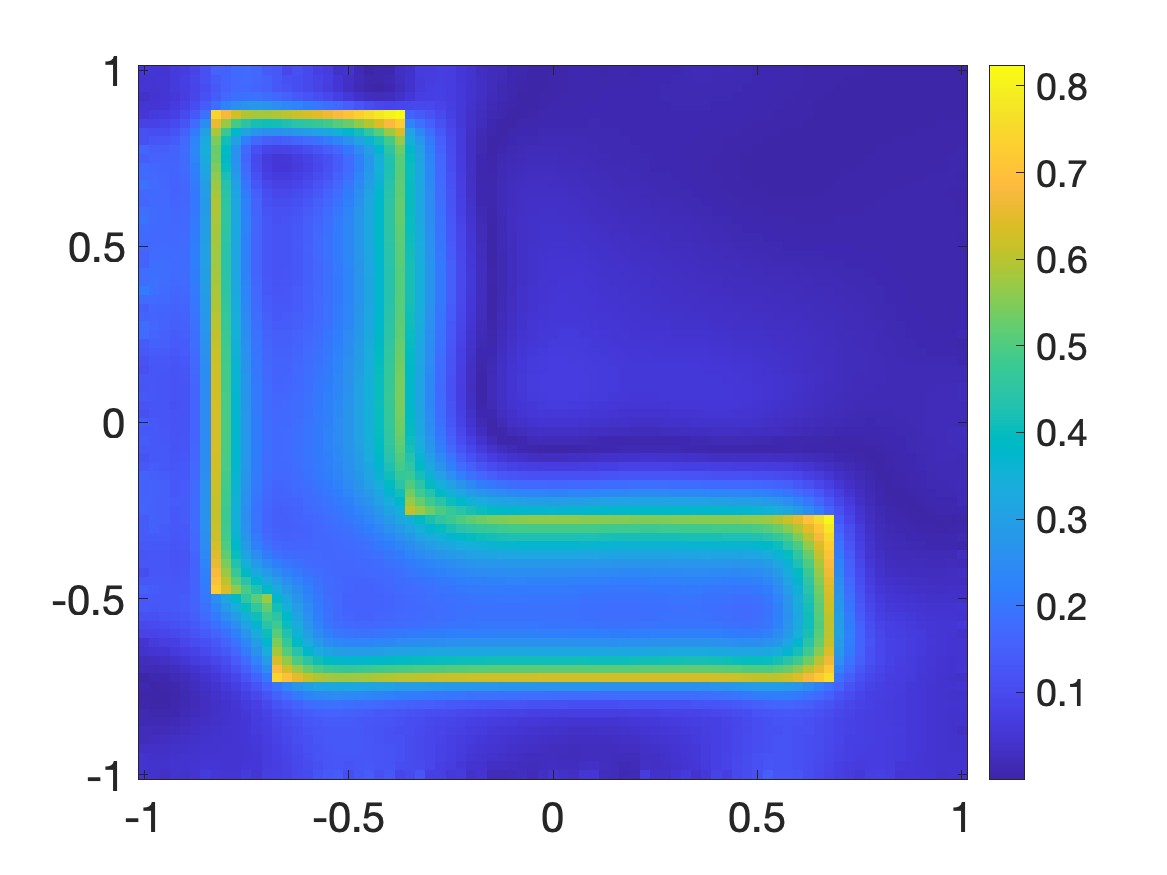}}
	\caption{\label{fig_test3} True and numerical solutions of test 3. 
	Like in tests 1 and 2, our method can deliver a satisfactory solution for test 2 without requesting a good initial guess. 
	Also, the error in computation occurs mostly at the boundary of the inclusion.
	}
	\end{center}
\end{figure}

Although the structure of the true source function is complicated and the value of the function is high (7), the numerical solution is out of expectation.
The maximum value of the reconstructed source function $g_{\rm comp}$ in the $L$ shape inclusion is 6.61 (the relative error is 5.5\%).

\section{Concluding remarks} \label{sec6}

The primary objective of this research paper is to tackle the task of computing initial conditions for quasi-linear and nonlocal hyperbolic equations. To achieve this goal, we propose the time dimensional reduction approach that involves approximating the solution of the hyperbolic equation through the truncation of its Fourier expansion in the time domain. By employing the polynomial-exponential basis, we effectively eliminate the time variable, resulting in a transformed system comprising quasi-linear elliptic equations. 
In order to globally solve this system without the requirement of a well-informed initial guess, we employ the powerful Carleman contraction principle. This principle allows us to navigate the intricacies of the system and derive solutions that meet our objectives. To substantiate the effectiveness of our proposed method, we provide a comprehensive range of numerical examples that showcase its efficacy in practice.

The noteworthy advantage of the time dimensional reduction method extends beyond its ability to deliver accurate solutions. It also boasts exceptional computational speed, which is a significant advantage in addressing complex problems efficiently.
We used an iMac with a processor of 3.2 GHz, Intel Core i5 built in 2015 to compute numerical solutions for the tests above. It took about
3.92 minutes including exporting the pictures to complete all tasks of Algorithm \ref{alg3}.
The speed is impressive since we can solve a nonlinear and nonlocal problem in $2 + 1$ dimension within a short time.

 \section*{Acknowledgement}
The works of LHN  were partially supported  by National Science Foundation grants DMS-2208159 and by funds provided by the Faculty Research Grant program at UNC Charlotte Fund No. 111272.
The authors extend their gratitude to VIASM and the organizers of the PDE and Related Topics workshop held at VIASM in Hanoi, Vietnam, in July 2022. The initiation of this project took place at this meeting.


\end{document}